\documentclass{sinst}

\usepackage{amsbsy,amsfonts,amsmath,amssymb,amsthm}
\usepackage{array}
\usepackage{bm}
\usepackage{color}
\usepackage{enumerate}
\usepackage{mathrsfs}
\usepackage{latexsym}
\usepackage[colorlinks=true]{hyperref}
\usepackage{mathtools}
\mathtoolsset{showonlyrefs, showmanualtags}
\hypersetup{urlcolor=blue, citecolor=blue, linkcolor=blue}

\definecolor{wineRed}{rgb}{0.7,0,0.3}
\definecolor{grandBleu}{rgb}{0,0,0.8}
\definecolor{darkGreen}{rgb}{0,0.4,0}
\definecolor{blueViolet}{rgb}{0.4,0,1.0}
\definecolor{bloodOrange}{rgb}{0.85,0.05,0}
\definecolor{mycolor}{rgb}{0.8,0,0.2}

\usepackage[textsize=small]{todonotes}
\usepackage{cite}

\usepackage{comment}

\DeclareMathAlphabet{\mathpzc}{OT1}{pzc}{m}{it}

\usepackage[textsize=small]{todonotes}
\numberwithin{equation}{section}

\theoremstyle{plain}
\newtheorem{mainThm}{Main Theorem} 
\newtheorem{lemma}{Lemma}[section]

\newtheorem{cor}{Corollary}

\theoremstyle{definition}
\newtheorem{definition}{Definition}
\newtheorem{rem}{Remark}

\newtheorem{ex}{Example}

\def\F{\mathcal{F}}
\def\N{\mathbb{N}}
\def\R{\mathbb{R}}

\def\L{\mathcal{L}}

\def\sH{\mathscr{H}}
\def\sF{\mathscr{F}}

\def\sV{\mathscr{V}}
\def\ds{\displaystyle}

\def\Lap{\mathit{\Delta}}
\def\Sgn{\mathop{\mathrm{Sgr}}\nolimits}

\DeclareMathOperator{\diver}{div}

\begin{document}
\vspace*{-0cm}
\title{\Large Well-posedness of parabolic KWC-systems with variable-dependent mobilities
\vspace{-2ex}
}

\author{Daiki Mizuno
}
\affiliation{Division of Mathematics and Informatics, \\ Graduate School of Science and Engineering, Chiba University, \\ 1-33, Yayoi-cho, Inage-ku, 263-8522, Chiba, Japan}
\email{d-mizuno@chiba-u.jp}
\vspace{-2ex}

\sauthor{Ken Shirakawa
}
\saffiliation{Department of Mathematics, Faculty of Education, Chiba University \\ 1--33 Yayoi-cho, Inage-ku, 263--8522, Chiba, Japan}
\semail{sirakawa@faculty.chiba-u.jp}
\vspace{-2ex}

\footcomment{
}
\maketitle

\noindent
{\bf Abstract.}
In this paper, we deal with the parabolic KWC system, associated with the mathematical model of grain boundary motion.
The goal of this paper is to guarantee the well-posedness of the parabolic KWC system. However, such results have not been reported under the setting where the mobility of grain boundary motion depends on the unknown.
To overcome this difficulty, results for the pseudo-parabolic type KWC system in [Antil et al., SIAM J. Math. Anal. \textbf{56}(5), 6422--6445](2024) suggest that the $H^1$-regularity of the time-derivative of the solution plays an essential role in verifying the uniqueness of the solution.
In this light, we consider the pseudo-parabolic KWC system as an approximating system of the parabolic one, and focus on the improvements of regularity of solution to the parabolic system.
By virtue of this regularity result, we establish the well-posedness theory on the parabolic KWC system.

\newpage 

\section{Introduction}

Let $N \in \{ 1,2,3,4 \}$ be a fixed dimension, and $\Omega \subset \R^N$ be a bounded domain with a boundary $\Gamma := \partial\Omega$.
When $N > 1$, we suppose that $\Gamma$ has $C^4$-regularity, and denote by $n_\Gamma$ the unit outer normal to $\Gamma$. $T > 0$ denotes a fixed time constant. Also, we set $Q_T := (0,T) \times \Omega$, $\Sigma_T := (0,T) \times \Gamma$.

In this paper, we consider a class of parabolic systems, denoted by (S)$_\varepsilon$ for $\varepsilon \in [0,1]$.
\begin{align}
  {\rm (S)}_\varepsilon:&
  \\[-4ex]
  &\left\{ \begin{aligned}
    &\partial_t \eta(t) - \Lap \eta(t) + g(\eta(t)) + \alpha'(\eta(t)) \sqrt{\varepsilon^2 + |\nabla \theta(t)|^2} = u(t), \mbox{ in } Q_T,
    \\
    &\nabla \eta(t) \cdot n_\Gamma = 0, \mbox{ on } \Sigma_T,
    \\
    &\eta(0,x) = \eta_0(x), \mbox{ for } x \in \Omega,
  \end{aligned} \right.
  \\[1.5ex]
  &\left\{ \begin{aligned}
    &\alpha_0(\eta(t)) \partial_t \theta(t) - \diver \left( \alpha(\eta(t)) \frac{\nabla \theta(t)}{\sqrt{\varepsilon^2 + |\nabla \theta(t)|^2}} + \kappa \nabla \theta(t) \right) = v(t), \mbox{ in } Q_T,
    \\
    &\left( \alpha(\eta(t)) \frac{\nabla \theta(t)}{\sqrt{\varepsilon^2 + |\nabla \theta(t)|^2}} + \kappa \nabla \theta(t) \right) \cdot n_\Gamma = 0, \mbox{ on } \Sigma_T,
    \\
    &\theta(0,x) = \theta_0(x), \mbox{ for } x \in \Omega.
  \end{aligned} \right.
\end{align}
Each system (S)$_\varepsilon$ is based on a mathematical model of grain boundary motion, proposed by \cite{MR1752970,MR1794359}. The unknowns of the system $\eta = \eta(t,x)$ and $\theta = \theta(t,x)$ are the order parameters to represent the orientation order and the orientation angle in a polycrystal body, respectively.
Their dynamics are governed by a free-energy for grain boundary motion, called \emph{KWC-energy}, defined by the following formula:
\begin{align}
  &\sF_\kappa^\varepsilon: [\eta, \theta] \in [L^2(\Omega)]^2 \mapsto \sF_\kappa^\varepsilon(\eta, \theta) 
  \\
  &\quad := \left\{ \begin{aligned}
    &\frac{1}{2} \int_\Omega |\nabla \eta|^2 \,dx + \int_\Omega G(\eta) \,dx + \Phi_\kappa^\varepsilon(\alpha(\eta);\theta), \mbox{ if } [\eta,\theta] \in [H^1(\Omega)]^2,
    \\
    &+\infty, \mbox{ otherwise}.
  \end{aligned} \right.
\end{align}
Here, for some nonnegative function $\beta \in H^1(\Omega)$, $\Phi_\kappa^\varepsilon(\beta;\cdot):L^2(\Omega) \longrightarrow [0,\infty]$ is a proper, lower semi-continuous and convex function to govern the reproduction of the grain boundaries, called \emph{interfacial energy}, and defined as:
\begin{align}
  &\Phi_\kappa^\varepsilon(\beta;\cdot): \ \theta \in L^2(\Omega) \mapsto \Phi_\kappa^\varepsilon(\beta;\theta)
  \\
  &\quad := \left\{ \begin{aligned}
    &\int_\Omega \beta \sqrt{\varepsilon^2 + |\nabla \theta|^2} \,dx + \frac{\kappa}{2} \int_\Omega |\nabla \theta|^2, \mbox{ if } \theta \in H^1(\Omega),
    \\
    &+\infty, \mbox{ otherwise}.
  \end{aligned} \right.
\end{align}
Then, the system (S)$_\varepsilon$ is derived as the following gradient flow.
\begin{equation}
  -\begin{bmatrix}
    \partial_t \eta(t)
    \\
    \alpha_0(\eta(t)) \partial_t \theta(t)
  \end{bmatrix} = \nabla_{[\eta, \theta]} \sF_\kappa^\varepsilon(\eta(t), \theta(t)) + \begin{bmatrix}
    u(t)
    \\
    v(t)
  \end{bmatrix} \mbox{ in } [L^2(\Omega)]^2, \ \mbox{ for a.e. } t \in (0,T).
\end{equation}
We note that when $\varepsilon = 0$, the free-energy $\sF_\kappa^0$ is non-smooth, and it results in the singular diffusivity of the second equation of (S)$_0$, which is mathematically justified by means of the subdifferential operator (cf. \cite{MR1814993,MR1799898,MR2033382,MR1712447}). 
In this context, $\kappa > 0$ is a fixed constant. $g = g(\eta)$ is a Lipschitz perturbation for the orientation order $\eta$, and $G = G(\eta)$ is a nonnegative primitive of $g$, i.e., $\frac{d}{d\eta}G = g$. $\alpha = \alpha(\eta)$ and $\alpha_0 = \alpha_0(\eta)$ are given nonnegative functions, called \emph{mobilities} of grain boundary motion. $u = u(t,x)$, $v = v(t,x)$ are given forcing terms. In particular, $u$ represents the relative temperature on $\Omega$. Finally, $\eta_0 = \eta_0(x)$, $\theta_0 = \theta_0(x)$ are the initial data for $[\eta, \theta]$.

The aim of this paper is to verify the well-posedness for the parabolic KWC system (S)$_\varepsilon$ under the setting that the mobility of grain boundary motion $\alpha_0$ depends on the order parameter $\eta$, which is a longstanding issue. 
Many researchers have been working to establish the mathematical validity of the KWC-type system. So far, the results regarding the existence of solution and its asymptotic behavior have been reported under various settings, such as homogeneous Neumann boundary condition (cf. \cite{MR2469586,MR2548486,MR2836555,MR3203495,MR3268865,MR3670006}), non-homogeneous Dirichlet boundary condition (cf. \cite{MR4352617}), dynamic boundary condition (cf. \cite{MR4677072}), time-periodic setting (cf. \cite{MR4683700}), and so on. Moreover, recent studies \cite{MR4603338,MR4809351} considered a 3D model of grain boundary motion, and provided the results of the existence of solution and its large-time behavior.

However, the guarantee of the uniqueness of the solution has been a challenging question, due to the velocity term $\alpha_0(\eta) \partial_t \theta$, and the singular diffusion $-\diver \bigl( \alpha(\eta) \frac{\nabla \theta}{|\nabla \theta|} \bigr)$, both of which involve the $\eta$-dependent mobilities. In fact, previous works which achieved the uniqueness \cite{MR2469586,MR2548486,MR2836555,MR3155454,MR4603338} modified $\alpha_0$ to be independent of $\eta$ (essentially, $\alpha_0$ is a constant).

Recently, \cite{MR4797677,MR4922530} reported a uniqueness result for a modified KWC system, derived as the following energy-dissipation flow with damping terms:
\begin{gather}
  -\begin{bmatrix}
    \partial_t \eta(t) - \mu^2 \Lap_N \partial_t \eta(t)
    \\
    \alpha_0(\eta(t)) \partial_t \theta(t) - \nu^2 \Lap_N \partial_t \theta(t)
  \end{bmatrix} = \nabla_{[\eta, \theta]} \sF_\kappa^\varepsilon(\eta(t), \theta(t)) + \begin{bmatrix}
    u(t)
    \\
    v(t)
  \end{bmatrix} \mbox{ in } [L^2(\Omega)]^2,
  \\
  \mbox{ for a.e. } t \in (0,T), \mbox{ and for arbitrary pair } [\mu,\nu] \in (0,1)^2, \label{E_PP_KWC}
\end{gather}
where $\Lap_N$ denotes the Laplacian operator subject to the zero-Neumann boundary condition. Such regularization is often referred to as the pseudo-parabolic regularization. The solution to \eqref{E_PP_KWC} is expected to have better regularity owing to the damping terms $-\Lap_N \partial_t \eta$ and $-\Lap_N \partial_t \theta$, such as:
\begin{equation}
  [\eta, \theta] \in [W^{1,2}(0,T;H^1(\Omega))]^2, \label{reg_PP_KWC}
\end{equation}
and this smoothing effect plays a key role to guarantee the uniqueness of the solution. Indeed, \cite{MR4797677,MR4922530} provided the proof of the uniqueness by using the strong regularity as in \eqref{reg_PP_KWC} together with the continuous embedding from $H^1(\Omega)$ to $L^4(\Omega)$, valid under $N \leq 4$. Therefore, this results suggest that the uniqueness assertion can be obtained if we achieve the strong regularity \eqref{reg_PP_KWC} for the solution to the parabolic KWC-type system (S)$_\varepsilon$.

On this basis, we consider the pseudo-parabolic KWC system as an approximating system for the parabolic one, and attempt to improve the regularity of solution through limiting process $\mu \downarrow 0$, $\nu \downarrow 0$. Along with this, we address the questions of the uniqueness and the continuous dependence of solution to the parabolic KWC system.

Based on the above discussion, the purpose of this paper is to derive strong regularity kindred to that obtained in pseudo-parabolic type PDEs for the solution to the parabolic KWC system, and to discuss the well-posedness of the parabolic KWC system with $\eta$-dependent mobility $\alpha_0$. More precisely, we treat the following two Main Theorems. 

\vspace{2ex}
\begin{description}
    \item[Main Theorem 1.] Improvement of regularity, and uniqueness of solution to (S)$_{\varepsilon_0}$ under smoothness conditions for initial data and forcing terms.
\vspace{1ex}
    \item[Main Theorem 2.] Continuous dependence of solution to (S)$_{\varepsilon_0}$ with respect to the initial data and forcings. 
\end{description}
\vspace{2ex}

Although certain additional assumptions are required, these Main Theorems provide a positive answer to the longstanding problem of the uniqueness of the solution for the parabolic KWC systems with the $\eta$-dependent mobility $\alpha_0$. Besides, the well-posedness results including the continuous dependence of solution with respect to the initial data and the forcings have potential applications to advanced mathematical topics, such as optimal control problems. Also, Main Theorem \ref{mainThm1} reinforces a fundamental assertion for regularity theory for nonlinear parabolic equations involving singularities; the smoothness of initial condition and forcings implies stronger regularity of the solution.

\medskip
\noindent
{\bf Outline:}
Preliminaries are given in the next Section 2, and on this basis, the Main Theorems are stated in Section 3. For the proofs of Main Theorems, we prepare Section 4 to set up an approximation method for (S)$_{\varepsilon_0}$ by means of pseudo-parabolicity. Based on these, the Main Theorems are proved in Section 5, by using the auxiliary results obtained in Section 4. 

\section{Preliminaries}
\label{sec:Prelim}
We begin by prescribing the notations used throughout this paper. 
\bigskip

\noindent
\underline{\textbf{\textit{Notations in real analysis.}}}
We define:
\begin{align*}
    & r \vee s := \max \{ r, s \} ~ \mbox{ and } ~ r \wedge s := \min \{r, s\}, \mbox{ for all $ r, s \in [-\infty, \infty] $,}
\end{align*}
and especially, we write:
\begin{align*}
    & [r]^+ := r \vee 0 ~ \mbox{ and } ~ [r]^- := -(r \wedge 0), \mbox{ for all $ r \in [-\infty, \infty] $.}
\end{align*}

Let $ d \in \N $ be a fixed dimension. We denote by $ |y| $ and $ y \cdot z $ the Euclidean norm of $ y \in \mathbb{R}^d $ and the scalar product of $ y, z \in \R^d $, respectively, i.e., 
\begin{equation*}
\begin{array}{c}
| y | := \sqrt{y_1^2 +\cdots +y_d^2} \mbox{ \ and \ } y \cdot z  := y_1 z_1 + \cdots + y_d z_d, 
\\[1ex]
\mbox{ for all $ y = [y_1, \ldots, y_d], ~ z = [z_1, \ldots, z_d] \in \mathbb{R}^d $.}
\end{array}
\end{equation*}
Besides, we let:
\begin{align*}
    & \mathbb{B}^d := \left\{ \begin{array}{l|l}
        y \in \R^d & |y| < 1
    \end{array} \right\} ~ \mbox{ and } ~ \mathbb{S}^{d -1} := \left\{ \begin{array}{l|l}
        y \in \R^d & |y| = 1
    \end{array} \right\}.
\end{align*}
We denote by $\mathcal{L}^{d}$ the $ d $-dimensional Lebesgue measure, and we denote by $ \mathcal{H}^{d} $ the $ d $-dimensional Hausdorff measure.  In particular, the measure theoretical phrases, such as ``a.e.'', ``$dt$'', and ``$dx$'', and so on, are all with respect to the Lebesgue measure in each corresponding dimension. Also on a Lipschitz-surface $ S $, the phrase ``a.e.'' is with respect to the Hausdorff measure in each corresponding Hausdorff dimension. In particular, if $S$ is $C^1$-surface, then we simply denote by $dS$ the area-element of the integration on $S$.

For a Borel set $ E \subset \R^d $, we denote by $ \chi_E : \R^d \longrightarrow \{0, 1\} $ the characteristic function of $ E $. Additionally, for a distribution $ \zeta $ on an open set in $ \R^d $ and any $i \in \{ 1,\dots,d \}$, let $ \partial_i \zeta$ be the distributional differential with respect to $i$-th variable of $\zeta$. As well as we consider, the differential operators, such as $\nabla,\ \diver, \ \nabla^2$, and so on, in distributional sense.
\bigskip

\noindent
\underline{\textbf{\textit{Abstract notations. (cf. \cite[Chapter II]{MR0348562})}}}
For an abstract Banach space $ X $, we denote by $ |\cdot|_{X} $ the norm of $ X $, and denote by $ \langle \cdot, \cdot \rangle_X $ the duality pairing between $ X $ and its dual $ X^* $. In particular, when $ X $ is a Hilbert space, we denote by $ (\cdot,\cdot)_{X} $ the inner product of $ X $. 

For two Banach spaces $ X $ and $ Y $,  let $  \mathscr{L}(X; Y)$ be the Banach space of bounded linear operators from $ X $ into $ Y $. 

For Banach spaces $ X_1, \dots, X_d $ with $ 1 < d \in \N $, let $ X_1 \times \dots \times X_d $ be the product Banach space endowed with the norm $ |\cdot|_{X_1 \times \cdots \times X_d} := |\cdot|_{X_1} + \cdots +|\cdot|_{X_d} $. However, when all $ X_1, \dots, X_d $ are Hilbert spaces, $ X_1 \times \dots \times X_d $ denotes the product Hilbert space endowed with the inner product $ (\cdot, \cdot)_{X_1 \times \cdots \times X_d} := (\cdot, \cdot)_{X_1} + \cdots +(\cdot, \cdot)_{X_d} $ and the norm $ |\cdot|_{X_1 \times \cdots \times X_d} := \bigl( |\cdot|_{X_1}^2 + \cdots +|\cdot|_{X_d}^2 \bigr)^{\frac{1}{2}} $. In particular, when all $ X_1, \dots,  X_d $ coincide with a Banach space $ Y $, the product space $X_1 \times \dots \times X_d$ is simply denoted by $[Y]^d$.
\bigskip

\noindent
\underline{\textbf{\textit{Basic notations.}}}
Let $0 < T < \infty$ be a fixed constant of time, and let $N \in \{ 1,2,3 \}$ is a fixed dimension. Let $\Omega \subset \R^N$ be a bounded domain with a boundary $\Gamma := \partial\Omega$, and when $N > 1$, $\Gamma$ has $C^\infty$-regularity with the unit outer normal $n_\Gamma$. Also, we define $d_\Gamma:\R \longrightarrow [0,\infty)$ as the distance function from $\Gamma$, i.e., for any $x \in \R^N$, $d_\Gamma(x)$ is the distance from $x$ to $\Gamma$. As is well-known, $n_\Gamma = \nabla d_\Gamma|_\Gamma$ on $\Gamma$, and hence, $\partial_i n_\Gamma = \nabla \partial_i d_\Gamma|_\Gamma$ on $\Gamma$ for all $i = 1, \ldots, N$. Moreover, due to the smoothness and the compactness of $\Gamma$, there exists a positive constant $C_\Gamma$ such that:
\begin{equation}
    |\partial_i n_\Gamma| = |\nabla \partial_i d_\Gamma| \leq C_\Gamma \mbox{ on } \Gamma, \mbox{ for all $ i = 1, \ldots, N $.} \label{eq:CGamma}
\end{equation}

Additionally, as notations of base spaces, we let:
\begin{equation*}
    H := L^2(\Omega), \ V := H^1(\Omega), ~ 
    \sH_T := L^2(0,T;H),  \mbox{ and } \sV_T := L^2(0,T;V).
\end{equation*}

\bigskip
\noindent
\underline{\textbf{\textit{Notations for the differential operators.}}}
Let $\Lap_N$ be the Laplace operator subject to the Neumann boundary condition. $\Lap_N$ can be regarded as a bounded linear operator from $V$ to $V^*$, with the following variational formulation:
\begin{equation}
    \langle \Lap_N z, w \rangle_V := -(\nabla z, \nabla w)_{[H]^N}, \ \mbox{ for all } z,w \in V. \label{eq:LapN}
\end{equation}
Besides, the distributional divergence $\diver$ can be also regarded as a bounded linear operator from $[H]^N$ to $V^*$, with the following variational identity:
\begin{equation}
    \langle \diver {\bm z}, w \rangle_V := -({\bm z}, \nabla w)_{[H]^N}, \ \mbox{ for any } {\bm z} \in [H]^N, \, \mbox{ and any } w \in V. \label{eq:div}
\end{equation}

Next, let $W_0 \subset H^2(\Omega)$ be the closed linear subspace of $H$, given by:
\begin{equation*}
  W_0 := \{ z \in H^2(\Omega)\,|\, \nabla z \cdot n_\Gamma = 0 \mbox{ on } \Gamma \}.
\end{equation*}
As is well-known that there exists a positive constant $ C_0 $ such that:
    \begin{align}\label{embb01}
        |z|_{H^2(\Omega)}^2 &\leq C_0 |-\Lap_N z + z|_H^2
        \\
        &\leq 2C_0 \bigl( |\Lap_N z|_H^2 + |z|_H^2 \bigr),~ \mbox{ for all $ z \in W_0 $.}
    \end{align}
\medskip

\begin{rem} \label{rem:resolvent}
    Let $z \in H$ be fixed, and let $\beta \in L^\infty(\Omega)$ be a function which fulfills $\delta_\beta := \inf \beta(\Omega) > 0$. For $\lambda \in (0,1]$, we define $w_\lambda \in W_0$ as the unique solution to the following elliptic problem:
    \begin{equation*}
        - \lambda \Lap_N w_\lambda + \beta w_\lambda = z \mbox{ in } H.
    \end{equation*}
    Then, the following estimate holds:
    \begin{equation}
        \frac{\delta_\beta}{2} |w_\lambda|_H^2 + \lambda |\nabla w_\lambda|_{[H]^N}^2 \leq \frac{1}{2 \delta_\beta} |z|_H, \mbox{ for every } \lambda \in (0,1]. \label{eq:elliptic}
    \end{equation}
    Moreover, if $z \in V$ and $\beta \in W^{1,\infty}(\Omega)$, then one can be checked that:
    \begin{equation}
        \frac{\delta_\beta}{2} |\nabla w_\lambda|_{[H]^N}^2 + \lambda |\Lap w_\lambda|_{[H]^N}^2 \leq \frac{|\nabla \beta|_{[L^\infty(\Omega)]^N}^2 + 1}{\delta_\beta^3 \wedge 1} |z|_V, \mbox{ for every } \lambda \in (0,1]. \label{eq:ellipticV}
    \end{equation}
    Moreover, if $\beta \equiv 1$, then the operator $(-\lambda \Lap_N + I)^{-1}$ is non-expansive in $H$, and in $V$.
\end{rem}

\medskip
\noindent
\underline{\textbf{\textit{Notations in convex analysis.}}}
Let $X$ be an abstract Hilbert space $X$. For a proper, lower semi-continuous (l.s.c.), and convex function $\Psi : \,X \longrightarrow (-\infty, \infty]$ on a Hilbert space $X$, we denote by $D(\Psi)$ the effective domain of $\Psi$. Also, we denote by $\partial \Psi$ the subdifferential of $\Psi$. The set $D(\partial \Psi) := \left\{ z \in X\,|\, \partial \Psi(z) \neq \emptyset \right\}$ is called the domain of $\partial\Psi$. The subdifferential $ \partial \Psi $ is known as a maximal monotone graph in the product space $X \times X$. We often use the notation ``$[z_0, z_0^*] \in \partial \Psi ~{\rm in}~ X \times X$", to mean that ``$z_0^* \in \partial \Psi(z_0) ~{\rm in}~ X~{\rm for}~ z_0 \in D(\partial \Psi)$", by identifying the operator $\partial \Psi$ with its graph in $X\times X$.
\medskip
\begin{ex}
    \label{exConvex}
    Let us define a convex function $\gamma_\varepsilon: \R^N \longrightarrow \R^N$, indexed by $\varepsilon \in [0,\infty)$, given as follows:
    \begin{equation}
        \gamma_\varepsilon:\, y \in \R^N \mapsto \gamma_\varepsilon(y) := \sqrt{\varepsilon^2 + |y|^2} \in [0,\infty).
    \end{equation}
    When $\varepsilon = 0$, $\gamma_0$ coincides with the Euclidean norm, and if $\varepsilon > 0$, then $\gamma_\varepsilon$ is a smooth approximation of the Euclidean norm. Regarding the subdifferential of the convex function $\gamma_\varepsilon$, the following items hold.
    \begin{description}
        \item[(O)]  The subdifferential $ \partial \gamma_0 \subset \R^N \times \R^N $ of the convex function $ \gamma_0 $, which coincides with the Euclidean norm, is given as the following set-valued function $ \Sgn: \R^N \rightarrow 2^{\mathbb{R}^N} $, defined by:
\begin{align}\label{Sgn^d}
\Sgn :  y = [y_1, & \dots, y_N] \in \mathbb{R}^N \mapsto \Sgn(y) = \Sgn(y_1, \dots, y_N) 
  \nonumber
  \\
  & := \left\{ \begin{array}{ll}
          \multicolumn{2}{l}{
                  \ds \frac{y}{|y|} = \frac{[y_1, \dots, y_N]}{\sqrt{y_1^2 +\cdots +y_N^2}}, ~ } \mbox{if $ y \ne 0 $,}
                  \\[3ex]
          \overline{\mathbb{B}^N}, & \mbox{otherwise.}
      \end{array} \right.
  \end{align}
\item[(\,I\,)]For every $ \varepsilon > 0 $, the subdifferential $\partial \gamma_\varepsilon$ is identified with the (single-valued) usual gradient, i.e.:
\begin{equation}
    D(\partial \gamma_\varepsilon) = \R^N \mbox{ and }\nabla \gamma_\varepsilon : \R^N \ni y \mapsto \nabla \gamma_\varepsilon(y) := \frac{y}{\sqrt{\varepsilon^2 + |y|^2}} \in \R^N.
\end{equation}
Moreover, since:
    \begin{align*}
        \gamma_\varepsilon(y) = \bigl| [\varepsilon, y] \bigr|_{\R^{N +1}} & =\bigl| [\varepsilon, y_1, \dots, y_N] \bigr|_{\R^{N +1}}, \mbox{ for all $ [\varepsilon, y] = [\varepsilon, y_1, \dots, y_N] \in \R^{N +1} $,}
        \\
        & \mbox{with $ \varepsilon \geq 0 $ and $ y = [y_1, \dots, y_N] \in \R^N $,}
    \end{align*}
    it will be estimated that:
    \begin{equation} 
            \bigl| \nabla \gamma_\varepsilon(y) \bigr|_{\R^{N}} = \left| \frac{y}{\bigl| [\varepsilon, y] \bigr|_{\mathbb{R}^{N +1}}} \right|_{\R^{N}} \leq \left| \frac{[\varepsilon, y]}{\bigl| [\varepsilon, y] \bigr|_{\mathbb{R}^{N +1}}} \right|_{\R^{N +1}} = 1, \mbox{ for } \varepsilon > 0,  \ \mbox{ and  for } y \in \R^N. \label{exM01}
    \end{equation}
\end{description}
\end{ex}

\medskip
\noindent
\underline{\textbf{\textit{Notion of Mosco-convergence. (cf. \cite{MR0298508})}}}
  Let $ X $ be an abstract Hilbert space. Let $ \Psi : X \rightarrow (-\infty, \infty] $ be a proper, l.s.c., and convex function, and let $ \{ \Psi_n \}_{n = 1}^\infty $ be a sequence of proper, l.s.c., and convex functions $ \Psi_n : X \rightarrow (-\infty, \infty] $, $ n = 1, 2, 3, \dots $.  Then, it is said that $ \Psi_n \to \Psi $ on $ X $, in the sense of Mosco, as $ n \to \infty $, iff. the following two conditions are fulfilled:
  \begin{description}
    \item[(\hypertarget{M_lb}{M1}) The condition of lower-bound:]$ \ds \varliminf_{n \to \infty} \Psi_n(\check{w}_n) \geq \Psi(\check{w}) $, if $ \check{w} \in X $, $ \{ \check{w}_n  \}_{n = 1}^\infty \subset X $, and $ \check{w}_n \to \check{w} $ weakly in $ X $, as $ n \to \infty $. 
    \item[(\hypertarget{M_opt}{M2}) The condition of optimality:]for any $ \hat{w} \in D(\Psi) $, there exists a sequence \linebreak $ \{ \hat{w}_n \}_{n = 1}^\infty  \subset X $ such that $ \hat{w}_n \to \hat{w} $ in $ X $ and $ \Psi_n(\hat{w}_n) \to \Psi(\hat{w}) $, as $ n \to \infty $.
  \end{description}
  As well as, if the sequence of convex functions $ \{ \widehat{\Psi}_\varepsilon \}_{\varepsilon \in \Xi} $ is labeled by a continuous argument $\varepsilon \in \Xi$ with a range $\Xi \subset \mathbb{R}$ , then for any $\varepsilon_{0} \in \Xi$, the Mosco-convergence of $\{ \widehat{\Psi}_\varepsilon \}_{\varepsilon \in \Xi}$, as $\varepsilon \to \varepsilon_{0}$, is defined by those of subsequences $ \{ \widehat{\Psi}_{\varepsilon_n} \}_{n = 1}^\infty $, for all sequences $\{ \varepsilon_n \}_{n=1}^{\infty} \subset \Xi$, satisfying $\varepsilon_{n} \to \varepsilon_{0}$ as $n \to \infty$.

\begin{rem}\label{Rem.MG}
  (cf. \cite[Theorem 3.66]{MR0773850} and \cite[Chapter 2]{Kenmochi81}) Let $ X $, $ \Psi $, and $ \{ \Psi_n \}_{n = 1}^\infty $ be as stated above. Let us assume that
    \begin{equation}\label{Mosco01}
      \Psi_n \to \Psi \mbox{ on $ X $, in the sense of Mosco, as $ n \to \infty $,}
      \vspace{-1ex}
    \end{equation}
and
\begin{equation*}
\left\{ ~ \parbox{10cm}{
$ [w, w^*] \in X \times X $, ~ $ [w_n, w_n^*] \in \partial \Psi_n $ in $ X \times X $, $ n \in \N $,
\\[1ex]
$ w_n \to w $ in $ X $ and $ w_n^* \to w^* $ weakly in $ X $, as $ n \to \infty $.
} \right.
\end{equation*}
Then, it holds that:
\begin{equation*}
[w, w^*] \in \partial \Psi \mbox{ in $ X \times X $, and } \Psi_n(w_n) \to \Psi(w) \mbox{, as $ n \to \infty $.}
\end{equation*}
\end{rem}

\begin{ex}[Examples of Mosco-convergence]\label{Rem.ExMG}
    Let $ \varepsilon_0 \geq 0 $ be arbitrary fixed constant and $\{ \gamma_\varepsilon \}_{\varepsilon \geq 0}$ be as in Example \ref{exConvex}. Then, the following items hold.
    \begin{description}
      \item[(O)] $\ds{\gamma_\varepsilon \to \gamma_{\varepsilon_0} \mbox{ on $ \R^N $, in the sense of Mosco, as $ \varepsilon \to \varepsilon_0 $.}}$
      \item[(\,I\,)] Let $0 \leq \beta \in H$ be a fixed function, and $\Phi_\kappa^\varepsilon(\beta;\cdot):H \longrightarrow [0,\infty]$ be a functional defined by:
      \begin{equation}
        w \in H \mapsto \Phi_\kappa^\varepsilon(\beta;w) := \left\{ \begin{aligned}
            &\int_\Omega \beta \gamma_\varepsilon(\nabla w) \,dx + \frac{\kappa}{2}\int_\Omega |\nabla w|^2\,dx, \mbox{ if } w \in V,
            \\
            & + \infty, \ \mbox{otherwise}.
        \end{aligned} \right.
      \end{equation}
      Here, let $\{ \beta_\varepsilon \}_{\varepsilon \in [0,\infty)} \subset H$ be a sequence of nonnegative functions such that $\beta_\varepsilon \to \beta_{\varepsilon_0}$ in $H$ as $\varepsilon \to \varepsilon_0$. Then it follows that
      \begin{equation}
        \Phi_\kappa^\varepsilon(\alpha_\varepsilon^\circ;\cdot) \to \Phi_\kappa^{\varepsilon_0}(\alpha_{\varepsilon_0}^\circ;\cdot) \mbox{ in the sense of Mosco, as } \varepsilon \to \varepsilon_0.
      \end{equation}
    \end{description}
\end{ex}
\medskip

Finally, we mention about a mathematical theory of singular diffusion equation of elliptic type, which is established by \cite[Section 3]{aiki2023class}.

Let $z \in H$ be fixed, and $\beta \in V$ be a fixed nonnegative function. We consider the following elliptic equation, governed by the subdifferential of convex function $\Phi_\kappa^\varepsilon(\beta;\cdot)$:
\begin{equation}
  \partial \Phi_\kappa^\varepsilon(\beta; w) + w \ni z \mbox{ in } H, \label{singular1}
\end{equation}
where $\partial \Phi_\kappa^\varepsilon(\beta; \cdot)$ is the subdifferential of the convex function $w \in H \mapsto \Phi_\kappa^\varepsilon(\beta;w) \in [0,\infty]$. When $\varepsilon > 0$, the subdifferential $\partial \Phi_\kappa^\varepsilon(\beta;\cdot)$ is a single-valued operator, and \eqref{singular1} is equivalent to the following quasilinear boundary-value problem:
\begin{equation}
  \left\{ \begin{aligned}
    &-\diver \bigl( \beta \nabla \gamma_\varepsilon(\nabla w) + \kappa \nabla w \bigr) + w = z \mbox{ in } \Omega,
    \\
    &\nabla w \cdot n_\Gamma = 0 \mbox{ on } \Gamma.
  \end{aligned} \right.
\end{equation}
On the other hand, when $\varepsilon = 0$, \eqref{singular1} is the mathematical formulation of the following singular diffusion equation:
\begin{equation}
  \left\{ \begin{aligned}
    &-\diver \Bigl( \beta \frac{\nabla w}{|\nabla w|} + \kappa \nabla w \Bigr) + w = z \mbox{ in } \Omega,
    \\
    &\Bigl( \beta \frac{\nabla w}{|\nabla w|} + \kappa \nabla w \Bigr) \cdot n_\Gamma = 0 \mbox{ on } \Gamma.
  \end{aligned} \right.
\end{equation}

The solvability of \eqref{singular1} is confirmed by Minty's theorem, and the regularity of the solution is studied in \cite[Section 3]{aiki2023class} under the case that $\beta \in V \cap L^\infty(\Omega)$ and $\varepsilon > 0$. Via slight modifications of the arguments in \cite[Section 3]{aiki2023class}, one can have the following $H^2$-estimate of the solution $w \in W_0$ to \eqref{singular1}.

\begin{lemma} \label{lem:SD_est}
  There exists a positive constant $C_\kappa^* = C_\kappa^*(\kappa)$, independent of $\varepsilon$ such that t
  
  The unique solution $w$ to the elliptic problem \eqref{singular1} satisfies the following estimate:
  \begin{equation}
    |w|_{H^2(\Omega)}^2 = |(\partial \Phi_\kappa^\varepsilon(\beta;\cdot) + I)^{-1}z|_{H^2(\Omega)}^2 \leq C_\kappa^* (|z|_H^2 + |\beta|_V^2), \label{SD_est}
  \end{equation}
  where $C_\kappa = C_\kappa^*(\kappa)$ is a positive constant, independent of $\varepsilon$, given by:
  \begin{equation}
    C_\kappa^* := \frac{16N(C_0 + 1)^2(N C_\Gamma^2 C_{\rm tr}^4 + 1)}{\kappa^2 \wedge 1} \geq 1. \label{C_kappa}
  \end{equation}
  In this context, $C_0$ is the constant as in \eqref{embb01}, $C_\Gamma$ is as in \eqref{eq:CGamma} and $C_{\rm tr}$ is the constant of the continuous embedding from $H^1(\Omega)$ to $L^2(\Gamma)$.
\end{lemma}

To obtain the estimate \eqref{SD_est}, the following technical lemma plays a key role.

\begin{lemma} \label{singular_diff_H2}
  (cf. \cite[Lemma 3.2]{aiki2023class}) Let $\varepsilon > 0$ be fixed, and let $\beta \in V$ be a nonnegative function. Then, there exists a constant $C^* = C^*(\rho)$ such that for any $\rho \in \R$, it holds that:
  \begin{gather}
    (\diver (\beta \nabla\gamma_\varepsilon(\nabla v)), \Lap v)_H \geq - \rho |v|_{H^2(\Omega)}^2 - \frac{N (N C_\Gamma^2 C_{\rm tr}^4 + 1)}{\rho} |\beta|_V^2, 
    \\
    \mbox{ for any } v \in W_0. \label{SD_est_lem}
  \end{gather}
\end{lemma}

\section{Main results}
\label{sec:main}

In this paper, the main assertions are discussed under the following assumptions.
\vspace{2ex}

\begin{itemize}
  \item[(A1)] $\kappa > 0$ is a fixed constant.
  \vspace{1ex}
  \item[(A2)] The perturbation $g \in C^{0,1}(\R) \cap C^1(\R)$ has a nonnegative primitive $G \in C^1(\R)$.
  \vspace{1ex}
  \item[(A3)] $\alpha \in C^{1,1}(\R) \cap C^2(\R)$ is a convex function. $\alpha_0$ belongs to the class $W^{1,\infty}(\R) \cap C^1(\R)$, and we suppose:
  \begin{equation*}
    \delta_{\alpha} := \inf \alpha_0(\R) > 0.
  \end{equation*}
  \item[(A4)] The forcings $u$ and $v$ belong to $W^{1,2}(0,T;H) \cap L^\infty(0,T;V)$ and $W^{1,2}(0,T;H)$, respectively.
    \vspace{1ex}
  \item[(A5)] The initial data $[\eta_0, \theta_0]$ of $[\eta, \theta]$ belong to the class:
  \begin{equation}
    [\eta_0, \theta_0] \in [W_0 \cap H^3(\Omega)] \times D(\partial \Phi_\kappa^{\varepsilon_0}(\alpha(\eta_0);\cdot)) =: W_*^{\varepsilon_0}.
  \end{equation}
\end{itemize}

Next, let us give the definition of the solution to (S)$_{\varepsilon_0}$.

\begin{definition} \label{def:sol}
  A pair of functions $[\eta, \theta] \in [\sH_T]^2$ is called a solution to (S)$_{\varepsilon_0}$ if and only if $[\eta,\theta]$ fulfills the following conditions (S0)--(S4).
  \begin{itemize}
    \item[(S0)] $[\eta, \theta] \in [W^{1,2}(0,T;H) \cap L^\infty(0,T;V) \cap L^2(0,T;W_0)] \times [W^{1,2}(0,T;H) \cap L^\infty(0,T;V)]$.
    \item[(S1)] $\eta$ solves the following identity:
    \begin{gather*}
      \bigl(\partial_t \eta(t) + g(\eta(t)) + \alpha'(\eta(t))\gamma_{\varepsilon_0}(\nabla \theta(t)), \varphi \bigr)_H + (\nabla \eta(t), \nabla \varphi)_{[H]^N} = (u(t), \varphi)_H,
      \\
      \mbox{ for any } \varphi \in V, \ \mbox{ and for a.e. } t \in (0,T).
    \end{gather*} 
    \item[(S2)] $\theta$ solves the following variational inequality:
    \begin{align*}
        &\bigl((\alpha_0(\eta) \partial_t \theta)(t), \theta(t) - \psi\bigr)_H + \Phi_\kappa^{\varepsilon_0}(\alpha(\eta(t));\theta(t)) 
        \\
        &\qquad \leq \Phi_\kappa^{\varepsilon_0}(\alpha(\eta(t));\psi) + (v(t), \theta(t) - \psi)_H, \mbox{ for any } \psi \in V, \mbox{ and a.e. }t \in (0,T).
    \end{align*}
    \item[(S3)] $[\eta(0),\theta(0)] = [\eta_0, \theta_0]$ in $[H]^2$.
    \item[(S4)] $[\eta, \theta]$ satisfies the following energy-inequality:
    \begin{align}
      &\frac{1}{4}\int_s^t |\partial_t \eta(r)|_H^2 \,dr + \frac{\delta_\alpha}{2} \int_s^t |\partial_t \theta(r)|_H^2 \,dr + \F_\kappa^{\varepsilon_0}(\eta(t), \theta(t)) 
      \\
      &\quad \leq \F_\kappa^{\varepsilon_0}(\eta(s),\theta(s)) + \frac{1}{2} \int_s^t |u(r)|_H^2 \,dr + \frac{1}{2 \delta_\alpha} \int_s^t |v(r)|_H^2 \,dr,
      \\
      &\qquad \mbox{ for any } s,t \in [0,T]; \ s < t.
    \end{align}
  \end{itemize}
\end{definition}

In previous studies \cite{MR2469586,MR2548486,MR2836555,MR3155454}, the solvability of (S)$_{\varepsilon_0}$ has been confirmed. Our principal dedication of this paper lies in the improvement of regularity of and the guarantee of the uniqueness of solution.

\begin{mainThm}[Existence, uniqueness and regualrity of solution] \label{mainThm1}
  Let $\varepsilon \in [0,1]$ be fixed. Under the assumptions (A1)--(A5), the following two statements hold.
  \begin{itemize}
    \item[(I)] The regularity of the solution $[\eta,\theta]$ is improved as follows:
    \begin{align}
      &\eta \in W^{2,2}(0,T;H) \cap W^{1,\infty}(0,T;V) \cap L^\infty(0,T;W_0),
      \\
      &\theta \in W^{1,\infty}(0,T;H) \cap W^{1,2}(0,T;V) \cap L^\infty(0,T;W_0).
    \end{align}
    \item[(II)] Let $[\eta^i,\theta^i]$ $(i=1,2)$ be two solutions to (S)$_{\varepsilon_0}$, corresponding to the same initial data $[\eta_0, \theta_0]$ and the same forcing pair $[u,v]$. Besides, we set
    \begin{align}
      &J(t) := |(\eta^1 - \eta^2)(t)|_H^2 + |\sqrt{\alpha_0(\eta^1(t))}(\theta^1 - \theta^2)(t)|_H^2, \ \mbox{ for any } t \in [0,T],
      \\
      &R(t) := |\partial_t \eta^1(t)|_V^2 + |\partial_t \theta^2(t)|_V^2 + 1, \ \mbox{ for a.e. } t \in (0,T),
    \end{align}
    and
    \begin{equation}
      C_1 := \frac{4}{(\kappa \wedge 1)(\delta_\alpha \wedge 1)} (|g'|_{L^\infty} + |\alpha'|_{L^\infty}^2 + (C_V^{L^4})^4|\alpha_0'|_{L^\infty}^2 + \kappa).
    \end{equation}
    Then, the following Gronwall type inequality holds:
    \begin{gather}
      \frac{1}{2} \frac{d}{dt}J(t) + |\nabla (\eta^1 - \eta^2)(t)|_{[H]^N}^2 + \frac{\kappa}{2} |\nabla (\theta^1 - \theta^2)(t)|_{[H]^N}^2 \leq C_1 R(t) J(t),
      \\
      \mbox{ for a.e. } t \in (0,T).
    \end{gather}
    Consequently, the solution to (S)$_{\varepsilon_0}$ is unique.
  \end{itemize}
\end{mainThm}

\begin{rem}
  When $\varepsilon_0 > 0$, standard variational methods allow us to rewrite the variational inequality of (S2) by:
  \begin{gather}
    \bigl( (\alpha_0(\eta) \partial_t \theta)(t), \psi \bigr)_H + \bigl( \alpha(\eta(t)) \nabla \gamma_{\varepsilon_0} (\nabla \theta(t)) + \kappa \nabla \theta(t), \nabla \psi \bigr)_{[H]^N} = (v(t), \psi)_H,
    \\
    \mbox{ for any } \psi \in V, \ \mbox{ and for a.e. } t \in (0,T).
  \end{gather}
  Also, the variational identity (S1) is equivalent to the following evolution equation:
  \begin{equation}
    \partial_t \eta(t) - \Lap_N \eta(t) + g(\eta(t)) + \alpha'(\eta(t)) \gamma_{\varepsilon_0}(\nabla \theta(t)) = u(t), \mbox{ in } H, \ \mbox{ for a.e. } t \in (0,T). \label{ev_eta}
  \end{equation}
  On the other hand, the variational inequality (S2) is equivalent to the following evolution equation:
  \begin{equation}
    \alpha_0(\eta(t)) \partial_t \theta(t) + \partial \Phi_\kappa^{\varepsilon_0} (\alpha(\eta(t)); \theta(t)) \ni v(t), \mbox{ in } H, \ \mbox{ for a.e. } t \in (0,T), \label{ev_theta}
  \end{equation}
  governed by time-dependent subdifferential operator of a convex function $z \in H \mapsto \Phi_\kappa^{\varepsilon_0} (\alpha(\eta(t)); z) \in [0,\infty]$. 
\end{rem}

\begin{mainThm}[Continuous dependence of solution] \label{mainThm2}
  Let $\{ \varepsilon_n \}_{n=1}^\infty \subset [0,\infty)$, \linebreak
  $\{ [\eta_{0,n}, \theta_{0,n}] \}_{n=1}^\infty \subset W_*^{\varepsilon_n}$ along with $\theta_n^* \in \partial \Phi_\kappa^{\varepsilon_0}(\alpha(\eta_{0,n});\theta_{0,n})$ for $n \in \N$, and $\{ [u_n, v_n] \}_{n=1}^\infty$ satisfy the following convergences:
  \begin{align}
    \left\{ \begin{aligned}
      &\varepsilon_n \to \varepsilon_0, \ \eta_{0,n} \to \eta_0 \mbox{ in } W_0 \mbox{ and weakly in } H^3(\Omega),
      \\
      &\theta_{0,n} \to \theta_0 \mbox{ in } V, \mbox{ and } \theta_n^* \to \theta_0^* \mbox{ weakly in } H,
      \\
      &u_n \to u \mbox{ weakly in } W^{1,2}(0,T;H) \mbox{ and weakly-$*$ in } L^\infty(0,T;V), \mbox{ and}
      \\
      &v_n \to v \mbox{ weakly in } W^{1,2}(0,T;H),
    \end{aligned} \right. \mbox{ as } n \to \infty. \label{mainThm2_01}
  \end{align}
  Let $[\eta_n,\theta_n], \, (n=1,2,3,\dots)$ be the unique solution to (S)$_{\varepsilon_n}$ corresponding to the initial data $[\eta_{0,n}, \theta_{0,n}]$ and the forcing term $[u_n, v_n]$, and let $[\eta, \theta]$ be the unique solution to (S)$_{\varepsilon_0}$ with the initial data $[\eta_0, \theta_0]$ and forcing pair $[u,v]$. Then, under the assumptions (A1)--(A5), the following convergences are obtained:
  \begin{gather}
    \left\{ \begin{aligned}
      &\eta_n \to \eta \mbox{ in } C([0,T];V),
      \\
      &\qquad \mbox{ weakly in } W^{2,2}(0,T;H), \, W^{1,q}(0,T;V) \mbox{ for any } q \in [1,\infty),
      \\
      &\qquad \mbox{ and weakly-$*$ in } L^\infty(0,T;W_0),
      \\
      &\partial_t \eta_n \to \partial_t \eta \mbox{ in } C([0,T];H), \mbox{ weakly in } W^{1,2}(0,T;H) \mbox{ and weakly-$*$ in } L^\infty(0,T;V),
      \\
      &\theta_n \to \theta \mbox{ in } C([0,T];V),
      \\
      &\qquad \mbox{ weakly in } W^{1,2}(0,T;V), \, W^{1,q}(0,T;H), \mbox{ for any } q \in [1,\infty),
      \\
      &\qquad \mbox{ and weakly-$*$ in } L^\infty(0,T;W_0),
      \\
      &\partial_t \theta_n \to \partial_t \theta \mbox{ weakly-$*$ in } L^\infty(0,T;H),
    \end{aligned} \right.
    \\
    \mbox{ as } n \to \infty. \label{mainThm2_02}
  \end{gather}
\end{mainThm}

\section{Approximating problem} \label{sec:approx}
In the proof of Main Theorem \ref{mainThm1}, the following system (S)$_\varepsilon^{\mu,\nu}$ plays a key role to approximate the original system (S):
\begin{subequations} \label{AP}
  \begin{align}
    &\quad \mbox{(S)}_\varepsilon^{\mu,\nu}:
    \\
    &\left\{ \begin{aligned}
        &\partial_t \eta -\mathit{\Delta} \bigl( \eta +\mu^2 \partial_t \eta \bigr) +g(\eta) +\alpha'(\eta) \gamma_\varepsilon(\nabla \theta) = u^\dagger(t, x), ~ \mbox{for $ (t, x) \in Q_{T} $,}
        \\[0.5ex]
        &\nabla \bigl( \eta +\mu^2 \partial_t \eta \bigr) \cdot n_\Gamma =  0, ~ \mbox{on $ \Sigma_{T} $,}
        \\[0.5ex]
        &\eta(0, x) = \eta_0^\dagger(x),~ \mbox{for $ x \in \Omega $,}
    \end{aligned} \right. \label{(AP)1st}
    \\[1.5ex]
    &\left\{ \begin{aligned}
        &\alpha_0(\eta) \partial_t \theta - \diver \left( \alpha(\eta) \nabla \gamma_\varepsilon(\nabla \theta) + \kappa \nabla \theta + \nu^2 \nabla \partial_t \theta \right) = v^\dagger(t, x), ~ \mbox{for $ (t, x) \in Q_{T} $,}
        \\[1ex]
        &\bigl( \alpha(\eta) \nabla \gamma_\varepsilon(\nabla \theta) + \kappa \nabla \theta + \nu^2 \nabla \partial_t \theta \bigr) \cdot n_\Gamma = 0 ~ \mbox{on $ \Sigma_{T} $,}
        \\[0.5ex]
        &\theta(0, x) = \theta_0^\dagger(x),~ \mbox{for $ x \in \Omega $.}
    \end{aligned} \right. \label{(AP)2nd}
  \end{align} 
\end{subequations}
The system (S)$_\varepsilon^{\mu,\nu}$ is derived by using pseudo-parabolic regularization to (S), which is given by the linear diffusion of the time-derivative of the solution $[\partial_t \eta, \partial_t \theta]$. In this context, $\mu > 0$, $\nu > 0$ are positive constants. $\kappa$, $g$, $\alpha$, $\alpha_0$ are given in the assumptions (A1)--(A3). $[u^\dagger, v^\dagger] \in [\sH_T]^2$ is a pair of forcings, and $[\eta_0^\dagger, \theta_0^\dagger] \in [W_0]^2$ is the initial pair of $[\eta, \theta]$. Also, throughout this subsection, we suppose that $\varepsilon$ is positive, and hence, $\gamma_\varepsilon$ is smooth. 

First, according to the studies of pseudo-parabolic singular diffusion equations (cf. \cite{aiki2023class}) and KWC systems of pseudo-parabolic type (cf. \cite{MR4797677,MR4922530}), we introduce a fundamental results, related to the well-posedness of the system (S)$_\varepsilon^{\mu,\nu}$, and regularity of solution.

\begin{lemma} \label{pseudo-parabolic}
  (cf. \cite{MR4797677,MR4922530}) Let $\mu \in (0,1), \, \nu \in (0,1)$, $[\eta_0^\dagger, \theta_0^\dagger] \in [V]^2$, and $u^\dagger, v^\dagger \in \sH_T$. Under the assumptions (A1)--(A3), (S)$_\varepsilon^{\mu,\nu}$ admits a unique solution $[\eta^\dagger, \theta^\dagger]$ in the following sense:
  \begin{itemize}
    \item[(S0)$_\varepsilon^{\mu,\nu}$] $[\eta^\dagger,\theta^\dagger] \in [W^{1,2}(0,T;V)]^2$,
    \item[(S1)$_\varepsilon^{\mu,\nu}$] $\eta^\dagger$ solves the following variational identity:
    \begin{gather*}
      \bigl( \partial_t \eta^\dagger(t) + g(\eta^\dagger(t)) + \alpha'(\eta^\dagger(t))\gamma_\varepsilon(\nabla \theta^\dagger(t)), \varphi \bigr)_H
      \\
      + (\nabla (\eta^\dagger + \mu^2 \partial_t \eta^\dagger)(t), \nabla \varphi )_{{[H]^N}} = (u^\dagger(t), \varphi)_H, \mbox{ for any } \varphi \in V, \mbox{ and a.e. } t \in (0,T).
    \end{gather*}
    \item[(S2)$_\varepsilon^{\mu,\nu}$] $\theta^\dagger$ solves the following variational identity:
    \begin{gather*}
        \bigl(\alpha_0(\eta^\dagger(t)) \partial_t \theta^\dagger(t), \psi\bigr)_H + \bigl( \alpha(\eta^\dagger(t)) \nabla\gamma_\varepsilon(\nabla \theta^\dagger(t)) + \kappa \nabla \theta^\dagger(t), \nabla \psi\bigr)_H 
        \\
        + \bigl( \nu^2 \nabla \partial_t \theta^\dagger(t), \nabla \psi \bigr)_{[H]^N}= (v^\dagger(t), \psi)_H, \mbox{ for any } \psi \in V, \mbox{ and a.e. }t \in (0,T).
    \end{gather*}
    \item[(S3)$_\varepsilon^{\mu,\nu}$] $[\eta^\dagger(0), \theta^\dagger(0)] = [\eta_0^\dagger, \theta_0^\dagger]$ in $[H]^2$.
    \item[(S4)$_\varepsilon^{\mu,\nu}$] $[\eta^\dagger, \theta^\dagger]$ satisfies the following energy-inequality:
    \begin{align}
      &\frac{1}{4}\int_s^t |\partial_t \eta^\dagger(r)|_H^2 \,dr + \mu^2 \int_s^t |\nabla \partial_t \eta^\dagger(r)|_{[H]^N}^2 \,dr + \frac{\delta_\alpha}{2} \int_s^t |\partial_t \theta^\dagger(r)|_H^2 \,dr
      \\
      &\quad \nu^2 \int_s^t |\nabla \partial_t \theta^\dagger|_{[H]^N}^2 \,dr + \F_\kappa^\varepsilon(\eta^\dagger(t), \theta^\dagger(t)) 
      \\
      &\quad \leq \F_\kappa^\varepsilon(\eta^\dagger(s),\theta^\dagger(s)) + \frac{1}{2} \int_s^t |u^\dagger(r)|_H^2 \,dr + \frac{1}{2 \delta_\alpha} \int_s^t |v^\dagger(r)|_H^2 \,dr,
      \\
      &\qquad \mbox{ for any } s,t \in [0,T]; \ s < t.
    \end{align}
  \end{itemize}
  Moreover, if $[\eta_0^\dagger, \theta_0^\dagger] \in [W_0]^2$, then it holds that
  \begin{equation}
    [\eta^\dagger, \theta^\dagger] \in W^{1,2}(0,T;W_0) \times L^\infty(0,T;W_0).
  \end{equation}
\end{lemma}

Lemma \ref{pseudo-parabolic} can be shown via some slight modifications of \cite{aiki2023class,MR4797677}, so we omit the proof.

The following Lemma will provide a fundamental estimates associated with $H^2$-regularity of solution $[\eta^\dagger, \theta^\dagger]$.
\begin{lemma}[Standard estimates] \label{lem:H2}
  There exist two positive constants $C_2 > 0$, $C_3 > 0$, which are independent of $\varepsilon$, $\mu$, $\nu$ such that the following estimates hold:
  \begin{align}
    &|\nabla \eta^\dagger(t)|_{[H]^N}^2 + \mu^2 |\Lap \eta^\dagger(t)|_H^2 + \frac{1}{2} \int_0^t |\Lap \eta^\dagger(s)|_H^2 \,ds \label{H2_01}
    \\
    &\quad \leq |\nabla \eta_0^\dagger|_{[H]^N}^2 + \mu^2 |\Lap \eta_0^\dagger|_H^2 + C_2(\varepsilon^2 + 1)\bigl( |\eta^\dagger|_{\sH_T}^2 + |\nabla \theta^\dagger|_{[\sH_T]^N}^2 + |u^\dagger|_{\sH_T}^2 + T \bigr), \mbox{ and} 
    \\
    &\nu^2 |\Lap \theta^\dagger(t)|_H^2 + \frac{\kappa}{2} \int_0^t |\Lap \theta^\dagger(s)|_H^2 \,ds \label{H2_02}
    \\
    &\quad \leq \nu^2 |\Lap \theta_0^\dagger|_H^2 + C_3(|\eta^\dagger|_{\sV_T}^2 + |\theta^\dagger|_{\sH_T}^2 + |\partial_t \theta^\dagger|_{\sH_T}^2 + |v^\dagger|_{\sH_T}^2 + T), \mbox{ for any } t \in [0,T]. 
  \end{align}
\end{lemma}

\begin{proof}[Proof of Lemma \ref{lem:H2}]
  Multiplying the both sides of \eqref{(AP)1st} by $-\Lap \eta^\dagger$ and using Young's and H\"{o}lder's inequalities, we can compute that
  \begin{align}
    &\frac{1}{2}\frac{d}{dt}\bigl( |\nabla \eta^\dagger(t)|_{[H]^N}^2 + \mu^2 |\Lap \eta^\dagger(t)|_H^2 \bigr) + \frac{1}{4}|\Lap \eta^\dagger(t)|_H^2 
    \\
    &\quad \leq |g(\eta^\dagger(t))|_H^2 + |\alpha'(\eta^\dagger(t))\gamma_\varepsilon(\nabla \theta^\dagger(t))|_H^2 + |u^\dagger(t)|_H^2, \label{H2_03}
    \\
    &\quad \leq 2 (|g(0)|^2 \L^N(\Omega) + |g'|_{L^\infty}^2 |\eta^\dagger(t)|_H^2) + |\alpha'|_{L^\infty}^2 (\varepsilon^2 \L^N(\Omega) + |\nabla \theta^\dagger(t)|_{[H]^N}^2) + |u^\dagger(t)|_H^2
    \\
    &\quad \leq 2 (\L^N(\Omega) + 1)(|g(0)|^2 + |g'|_{L^\infty}^2 + |\alpha'|_{L^\infty}^2 + 1) \cdot
    \\
    &\qquad \quad \cdot (\varepsilon^2 + 1) (|\eta^\dagger(t)|_H^2 + |\nabla \theta^\dagger(t)|_{[H]^N}^2 + |u^\dagger(t)|_H^2 + 1), \ \mbox{ for a.e. } t \in (0,T).
  \end{align}
  Hence, by integrating the both sides of \eqref{H2_03} over $[0,t]$, we arrive at the inequality \eqref{H2_01} with:
  \begin{equation}
    C_2 := 4 (\L^N(\Omega) + 1)(|g(0)|^2 + |g'|_{L^\infty}^2 + |\alpha'|_{L^\infty}^2 + 1).
  \end{equation}
  
  On the other hand, by taking the inner product of \eqref{(AP)2nd} with $-\Lap \theta^\dagger$, one can say that
  \begin{gather}
    \frac{\nu^2}{2} \frac{d}{dt}\bigl( |\Lap \theta^\dagger(t)|_H^2 \bigr) + \frac{\kappa}{2}|\Lap \theta^\dagger(t)|_H^2 + \bigl( \diver(\alpha(\eta^\dagger(t)) \nabla \gamma_\varepsilon(\nabla \theta^\dagger(t))), \Lap \theta^\dagger(t) \bigr)_H 
    \\
    \leq \frac{|\alpha_0|_{L^\infty}^2}{\kappa}|\partial_t \theta^\dagger(t)|_H^2 + \frac{1}{\kappa}|v(t)|_H^2, \mbox{ for a.e. } t \in (0,T). \label{H2_04}
  \end{gather}
  Applying Lemma \ref{singular_diff_H2} with $\rho := \frac{\kappa}{8 C_0}$ and setting $\widetilde{C}_4$ as:
  \begin{equation}
    \widetilde{C}_4 := \frac{8N (C_0 + 1) (N C_\Gamma^2 C_{\rm tr}^4 + 1)}{\kappa \wedge 1} \geq \frac{1}{\kappa \wedge 1}, \label{H2_04-1}
  \end{equation}
  we have:
  \begin{gather}
    \begin{aligned}
      &\bigl( \diver (\alpha(\eta^\dagger(t)) \nabla \gamma_\varepsilon(\nabla \theta^\dagger(t))), \Lap \theta^\dagger(t) \bigr)_H \geq - \frac{\kappa}{4} (|\Lap \theta^\dagger(t)|_H^2 + |\theta^\dagger(t)|_H^2) - \widetilde{C}_4 |\alpha(\eta^\dagger(t))|_V^2
      \\
      &\qquad \geq - \frac{\kappa}{4} (|\Lap \theta^\dagger(t)|_H^2 + |\theta^\dagger(t)|_H^2) - \widetilde{C}_4 (2|\alpha(0)|^2 \L^N(\Omega) + 2|\alpha'|_{L^\infty}^2 |\eta^\dagger|_V^2)
      \\
      &\qquad \geq - \frac{\kappa}{4} |\Lap \theta^\dagger(t)|_H^2 - 2 \widetilde{C}_4(|\alpha(0)|^2 \L^N(\Omega) + |\alpha'|_{L^\infty}^2 + \kappa + 1) (|\eta^\dagger|_V^2 + |\theta^\dagger(t)|_H^2 + 1),
    \end{aligned}
    \\
    \mbox{ for a.e. } t \in (0,T). \label{H2_05}
  \end{gather}
  Therefore, combining \eqref{H2_04} and \eqref{H2_05} yields that:
    \begin{align}
      &\quad \frac{\nu^2}{2} \frac{d}{dt} \bigl( |\Lap \theta^\dagger(t)|_H^2 \bigr) + \frac{\kappa}{4}|\Lap \theta^\dagger(t)|_H^2
      \\
      &\leq 2 \widetilde{C}_4(|\alpha(0)|^2 \L^N(\Omega) + |\alpha'|_{L^\infty}^2 + \kappa + 1) (|\eta^\dagger(t)|_V^2 + |\theta^\dagger(t)|_H^2 + 1) 
      \\
      &\qquad+ \frac{|\alpha_0|_{L^\infty}^2}{\kappa} |\partial_t \theta^\dagger(t)|_H^2 + \frac{1}{\kappa} |v(t)|_H^2
      \\
      &\leq 2 \widetilde{C}_4 (\kappa + 1) (|\alpha(0)|^2 \L^N(\Omega) + |\alpha'|_{L^\infty}^2 + |\alpha_0|_{L^\infty}^2 + 1) \cdot 
      \\
      &\qquad \cdot (|\eta^\dagger(t)|_V^2 + |\theta^\dagger(t)|_H^2 + |\partial_t \theta^\dagger(t)|_H^2 + |v^\dagger(t)|_H^2 + 1), \mbox{ for a.e. } t \in (0,T). \label{H2_06}
    \end{align}
  By integrating the both sides of \eqref{H2_06} over $[0,t]$, we obtain the estimate \eqref{H2_02} with the constant:
  \begin{align}
    C_3 &:= 2 \widetilde{C}_4 (\kappa + 1) (|\alpha(0)|^2 \L^N(\Omega) + |\alpha'|_{L^\infty}^2 + |\alpha_0|_{L^\infty}^2 + 1)
    \\
    &= \frac{16N (\kappa + 1)(C_0 + 1) (N C_\Gamma^2 C_{\rm tr}^4 + 1)}{\kappa \wedge 1} (|\alpha(0)|^2 \L^N(\Omega) + |\alpha'|_{L^\infty}^2 + |\alpha_0|_{L^\infty}^2 + 1).
  \end{align}

\end{proof}

Next, we discuss the differentiability of the system (S)$_\varepsilon^{\mu,\nu}$ with respect to time-variable under a suitable differentiability condition for the forcings $[u^\dagger, v^\dagger]$.

\begin{lemma} \label{time-diff}
  We suppose $u^\dagger, v^\dagger \in W^{1,2}(0,T;H)$. Then, for $\mu > 0, \nu > 0$, the pair of functions $[p^\dagger,z^\dagger] := [\partial_t \eta^\dagger, \partial_t \theta^\dagger]$ solves the following variational system:
  \begin{equation}
    p^\dagger \in W^{1,2}(0,T;W_0), \ z^\dagger \in W^{1,2}(0,T;V), \label{time-diff01}
  \end{equation}
  \begin{gather}
    \bigl((\partial_t p^\dagger + g'(\eta^\dagger)p^\dagger + \alpha''(\eta^\dagger) p^\dagger \gamma_\varepsilon(\nabla \theta^\dagger) + \alpha'(\eta^\dagger) \nabla \gamma_\varepsilon(\nabla \theta^\dagger) \cdot \nabla z^\dagger)(t), \varphi \bigr)_H \label{time-diff02}
    \\
    + (\nabla (p^\dagger + \mu^2 \partial_t p^\dagger)(t), \nabla \varphi)_H = (u'(t), \varphi)_H, \ \mbox{ for any } \varphi \in V, \ \mbox{ and a.e. } t \in (0,T),
  \end{gather}
  and 
  \begin{gather}
    \bigl( (\alpha_0(\eta^\dagger) \partial_t z^\dagger + \alpha_0'(\eta^\dagger) p^\dagger z^\dagger)(t), \psi \bigr)_H + \bigl( (\alpha(\eta^\dagger) \nabla^2 \gamma_\varepsilon(\nabla \theta^\dagger) \nabla z^\dagger)(t), \nabla \psi \bigr)_{[H]^N}
    \\
    + \bigl( (\kappa \nabla z^\dagger + \nu^2 \nabla \partial_t z^\dagger + \alpha'(\eta^\dagger ) p^\dagger \nabla \gamma_\varepsilon(\nabla \theta^\dagger))(t), \nabla \psi \bigr)_{[H]^N} = (v'(t), \psi)_H, \label{time-diff03}
    \\
    \mbox{ for any } \psi \in V, \ \mbox{ and a.e. } t \in (0,T),
  \end{gather}
  with the initial condition:
  \begin{align}
    p^\dagger(0) &= (I - \mu^2 \Lap_N)^{-1}\bigl( \Lap_N \eta_0^\dagger - g(\eta_0^\dagger) - \alpha'(\eta_0^\dagger) \gamma_\varepsilon(\nabla \theta_0^\dagger) + u(0) \bigr) \in W_0, \mbox{ and } 
    \\
    z^\dagger(0) &= (\alpha_0(\eta_0^\dagger) I - \nu^2 \Lap_N)^{-1} \bigl( -\partial \Phi_\kappa^\varepsilon(\alpha(\eta_0);\theta_0^\dagger) + v(0) \bigr)
    \\
    &= (\alpha_0(\eta_0^\dagger) I - \nu^2 \Lap_N)^{-1} \bigl( \diver \bigl( \alpha(\eta_0^\dagger) \nabla \gamma_\varepsilon(\nabla \theta_0^\dagger) + \kappa \nabla \theta_0^\dagger \bigr) + v(0) \bigr) \in V.
  \end{align}
\end{lemma}

\begin{proof}
  We set $[p^\ddagger,z^\ddagger]$ as the unique solution to the following pseudo-parabolic linear system:
  \begin{align}
    &\left\{ \begin{aligned}
      &\partial_t p^\ddagger(t) - \Lap_N (p^\ddagger + \mu^2 \partial_t p^\ddagger)(t) = h^\dagger(t), \ \mbox{ in } H, \mbox{ for a.e. } t \in (0,T),
      \\
      &p^\ddagger(0) = (I - \mu^2 \Lap_N)^{-1}\bigl( \Lap_N \eta_0^\dagger - g(\eta_0^\dagger) - \alpha'(\eta_0^\dagger) \gamma_\varepsilon(\nabla \theta_0^\dagger) + u(0) \bigr), \mbox{ and }
    \end{aligned} \right. \label{time-diff04}
    \\
    &\left\{ \begin{aligned}
      &\bigl( \alpha_0(\eta^\dagger)\partial_t z^\ddagger + \alpha_0'(\eta^\dagger) \partial_t \eta^\dagger z^\dagger \bigr)(t) - \Lap_N (\kappa z^\ddagger + \nu^2 \partial_t z^\ddagger)(t) = k^\dagger(t), \ \mbox{ in } V^*, \mbox{ a.e. } t \in (0,T),
      \\
      &z^\ddagger(0) = (\alpha_0(\eta_0^\dagger) I - \nu^2 \Lap_N)^{-1} \bigl( -\partial \Phi_\kappa^\varepsilon(\alpha(\eta_0);\theta_0^\dagger) + v(0) \bigr), \label{time-diff05}
    \end{aligned} \right.
  \end{align}
  with
  \begin{align}
    h^\dagger &:= \partial_t \bigl( - g(\eta^\dagger) - \alpha'(\eta^\dagger) \gamma_\varepsilon(\nabla \theta^\dagger) + u \bigr) \label{time-diff07}
    \\
    &= -g'(\eta^\dagger) \partial_t \eta^\dagger - \alpha''(\eta^\dagger) \partial_t \eta^\dagger \gamma_\varepsilon(\nabla \theta^\dagger) - \alpha'(\eta^\dagger) \nabla \gamma_\varepsilon(\nabla \theta^\dagger) \cdot \nabla \partial_t \theta^\dagger + u' \in \sH_T, \mbox{ and}
    \\
    k^\dagger &:= \partial_t \bigl( \diver \bigl( \alpha(\eta^\dagger) \nabla \gamma_\varepsilon(\nabla \theta^\dagger) \bigr) + v \bigr) \label{time-diff08}
    \\
    &= \diver \bigl( \alpha(\eta^\dagger) \nabla^2 \gamma_\varepsilon(\nabla \theta^\dagger) \nabla \partial_t \theta^\dagger + \alpha'(\eta^\dagger) \partial_t \eta^\dagger \nabla \gamma_\varepsilon(\nabla \theta^\dagger) \bigr) + v' \in \sV_T^*.
  \end{align}
  Then, the theory of pseudo-parabolic linear PDEs (e.g. \cite{MR437936}) allows us to verify the existence of $[p^\dagger, z^\dagger]$ and the regularities \eqref{time-diff01}.

  Now, by setting
  \begin{equation}
    \eta^\ddagger(t) := \eta_0^\dagger + \int_0^t p^\dagger(s) \,ds \mbox{ in } W_0, \mbox{ and } \theta^\ddagger(t) := \theta_0^\dagger + \int_0^t z^\dagger(s) \,ds, \mbox{ in } V,
  \end{equation}
  one can confirm that $[\eta^\ddagger, \theta^\ddagger]$ solves the variational form (S1)$_\varepsilon^{\mu,\nu}$ and (S2)$_\varepsilon^{\mu,\nu}$ of (S)$_\varepsilon^{\mu,\nu}$. Since the solution to (S)$_\varepsilon^{\mu,\nu}$ is unique, we have $[\eta^\ddagger, \theta^\ddagger] = [\eta^\dagger, \theta^\dagger]$. Therefore, we see that $[p^\dagger, z^\dagger] = [p^\ddagger, z^\ddagger]$, and obtain the assertions of Lemma \ref{time-diff}.
\end{proof}
\section{Proofs of Main Theorems}
The proofs of Main Theorems are divided into four subsections, listed below.

\begin{itemize}
  \item[$\S$ 5.1] Boundedness of time-derivatives $\partial_t \eta_{\varepsilon,\mu,\nu}$ and $\partial_t \theta_{\varepsilon,\mu,\nu}$.
  \item[$\S$ 5.2] Limiting observations for some subsequence of $\{ [\eta_{\varepsilon,\mu,\nu}, \theta_{\varepsilon,\mu,\nu}] \}$.
  \item[$\S$ 5.3] Verification of (S0)--(S4).
  \item[$\S$ 5.4] Verification of the uniqueness of solution.
  \item[$\S$ 5.5] The proof of Main Theorem 2.
\end{itemize}

\subsection{Boundedness of time-derivatives $\partial_t \eta_{\varepsilon,\mu,\nu}$ and $\partial_t \theta_{\varepsilon,\mu,\nu}$} \label{sec:TD_bdd}
We begin to set the approximating sequences for the solution to the system (S)$_\varepsilon$ based on Lemma \ref{pseudo-parabolic} and \ref{time-diff}. First, we let $\{ \theta_{0,\varepsilon} \}_{\varepsilon \in (0,1)}$ be given as:
\begin{equation}
  \theta_{0,\varepsilon} := (\partial \Phi_\kappa^\varepsilon(\alpha(\eta_0);\cdot) + I)^{-1} (w^* + \theta_0), \, \mbox{ for some } w^* \in \partial \Phi_\kappa^{\varepsilon_0}(\alpha(\eta_0);\theta_0).
\end{equation}
First, we consider the convergence of sequence $\{ \theta_{0,\varepsilon} \}_{\varepsilon \in (0,1)}$ as $\varepsilon \to \varepsilon_0$.

\begin{lemma} \label{init}
  Under the assumptions (A1)--(A5), letting $\varepsilon \to \varepsilon_0$ implies the following two convergences:
  \begin{equation}
    \left\{ \begin{aligned}
      &\theta_{0,\varepsilon} \to \theta_0 \mbox{ in } V \mbox{ and weakly in } W_0,
      \\
      &w_\varepsilon^* := \partial \Phi_\kappa^\varepsilon(\alpha(\eta_0);\theta) \to w^* \mbox{ weakly in } H,
    \end{aligned} \right. \label{init_00}
  \end{equation}
\end{lemma}

\begin{proof}
  Owing to the estimate \eqref{SD_est}, we can obtain the following estimate:
  \begin{gather}
    |\theta_{0,\varepsilon}|_{H^2(\Omega)}^2 = |(\partial \Phi_\kappa^\varepsilon(\alpha(\eta_0);\cdot) + I)^{-1} (w^* + \theta_0)|_{H^2(\Omega)}^2 \leq C_\kappa^* (|w^* + \theta_0|_H^2 + |\alpha(\eta_0)|_V^2),
    \\
    \mbox{ for any } \varepsilon \in (0,1), \label{init_001}
  \end{gather}
  where $C_\kappa^*$ is the constant as in \eqref{SD_est}. Therefore, we can obtain the following boundedness:
  \begin{itemize}
    \item $\{ \theta_{0,\varepsilon} \}_{\varepsilon \in (0,1)}$ is bounded in $W_0$.
  \end{itemize}
  By applying the Rellich--Kondrachov theorem, we can find a subsequence $\{ \varepsilon_n \}_{n = 1}^\infty \subset (0,\infty)$, convergent to $\varepsilon_0$, with a limiting function $\tilde \theta \in W_0$ such that:
  \begin{equation}
    \theta_{0,\varepsilon_n} \to \tilde \theta \mbox{ in } V \mbox{ and weakly in } W_0 \mbox{ as } n \to \infty. \label{init_04}
  \end{equation}
  Moreover, noting the following identity:
  \begin{equation}
    \partial \Phi_\kappa^\varepsilon(\alpha(\eta_0);\theta_{0,\varepsilon}) = -\theta_{0,\varepsilon} + w^* + \theta_0 \mbox{ in } H,
  \end{equation}
  we obtain an additional convergence:
  \begin{equation}
    w_{\varepsilon_n}^* \to \tilde w := -\tilde \theta + w^* + \theta_0 \mbox{ in } H \mbox{ as } n \to \infty. \label{init_0011}
  \end{equation}

  Now, $\tilde \theta$ and $\tilde w$ satisfy the following identity:
  \begin{equation}
    \tilde w + \tilde \theta = w^* + \theta_0 \mbox{ in } H. \label{init_05}
  \end{equation}
  In addition, having Remark \ref{Rem.MG} and Example \ref{Rem.ExMG} (I) in mind, one can find that:
  \begin{equation}
    \tilde w \in \partial \Phi_\kappa^{\varepsilon_0}(\alpha(\eta_0);\tilde \theta) \mbox{ in } H.
  \end{equation}
  which implies that $\tilde \theta$ solves the following elliptic equation:
  \begin{equation}
    \bigl( \partial \Phi_\kappa^{\varepsilon_0}(\alpha(\eta_0); \cdot) + I \bigr) \tilde \theta \ni w^* + \theta_0 \mbox{ in } H. \label{init_06}
  \end{equation}
  Here, one can say that $\theta_0$ also solves the elliptic equation \eqref{init_06}. Hence, the uniqueness of solution to \eqref{init_06} leads that $\tilde \theta = \theta_0$ in $H$, and taking \eqref{init_05} into account, we have $\tilde w = w^*$. 

  Finally, by the uniqueness of limit, the convergence \eqref{init_00} and \eqref{init_001} can be verified. Thus, the proof of Lemma \ref{init} is completed. 

\end{proof}

Let $[\eta_{\varepsilon,\mu,\nu},\theta_{\varepsilon,\mu,\nu}]$ be the solution to (S)$_\varepsilon^{\mu,\nu}$ for $\varepsilon \in (0,1)$, $\mu \in (0,1)$ and $\nu \in (0,1)$, under the case when:
\begin{equation}
  [\eta_0^\dagger, \theta_0^\dagger] = [\eta_0, \theta_{0,\varepsilon}], \ \mbox{ and } [u^\dagger, v^\dagger] = [u,v].
\end{equation}

Next, we consider to derive some estimates for the time derivative of solution. In view of this, we set $[p_{\varepsilon,\mu,\nu},z_{\varepsilon,\mu,\nu}] := [\partial_t \eta_{\varepsilon,\mu,\nu}, \partial_t \theta_{\varepsilon,\mu,\nu}]$. In addition, for some positive constant $\sigma > 0$, we set $[\eta_{\varepsilon,\mu,\nu}^\circ,\theta_{\varepsilon,\mu,\nu}^\circ]$ by scale transformation in time, as follows:
\begin{equation}
  [\eta_{\varepsilon,\mu,\nu}^\circ,\theta_{\varepsilon,\mu,\nu}^\circ](t) := [\eta_{\varepsilon,\mu,\nu},\theta_{\varepsilon,\mu,\nu}](t/\sigma), \ \mbox{ and } [u^\circ, v^\circ] := [u,v](t/\sigma), \mbox{ for } t \in [0,\sigma T].
\end{equation}
Then, $[\eta_{\varepsilon,\mu,\nu}^\circ,\theta_{\varepsilon,\mu,\nu}^\circ]$ and its time derivative $[p_{\varepsilon,\mu,\nu}^\circ,z_{\varepsilon,\mu,\nu}^\circ] := [\partial_t \eta_{\varepsilon,\mu,\nu}^\circ,\partial_t \theta_{\varepsilon,\mu,\nu}^\circ]$ solve the following variational systems:
\begin{equation}
  [\eta_{\varepsilon,\mu,\nu}^\circ,\theta_{\varepsilon,\mu,\nu}^\circ] \in [W^{2,2}(0,\sigma T;V) \cap W^{1,2}(0,\sigma T;W_0)] \times [W^{2,2}(0,\sigma T;V) \times L^\infty(0,\sigma T;H^2(\Omega))],
\end{equation}
\begin{equation}
  \left\{ \begin{gathered}
    \begin{gathered}
      \bigl( \sigma \partial_t \eta_{\varepsilon,\mu,\nu}^\circ(t) + g(\eta_{\varepsilon,\mu,\nu}^\circ(t)) + \alpha'(\eta_{\varepsilon,\mu,\nu}^\circ(t))\gamma_\varepsilon(\nabla \theta_{\varepsilon,\mu,\nu}^\circ(t)), \varphi \bigr)_H + (\nabla \eta_{\varepsilon,\mu,\nu}^\circ(t), \nabla \varphi)_{[H]^N}
      \\
      + \sigma \mu^2(\nabla \partial_t \eta_{\varepsilon,\mu,\nu}^\circ(t), \nabla \varphi)_{{[H]^N}} = (u^\circ(t), \varphi)_H, \mbox{ for any } \varphi \in V, \mbox{ and a.e. } t \in (0,\sigma T).
    \end{gathered}
    \\[2ex]
    \begin{gathered}
        \bigl(\sigma (\alpha_0(\eta_{\varepsilon,\mu,\nu}^\circ) \partial_t \theta_{\varepsilon,\mu,\nu}^\circ)(t), \theta_{\varepsilon,\mu,\nu}^\circ(t) - \psi\bigr)_H + \bigl( \alpha(\eta_{\varepsilon,\mu,\nu}^\circ(t)) \nabla\gamma_\varepsilon(\nabla \theta_{\varepsilon,\mu,\nu}^\circ(t)) + \kappa \nabla \theta_{\varepsilon,\mu,\nu}^\circ, \nabla \psi\bigr)_H
        \\
        + \bigl( \sigma \nu^2 \nabla \partial_t \theta_{\varepsilon,\mu,\nu}^\circ(t), \nabla \psi \bigr)_{[H]^N}= (v^\circ(t), \psi)_H, \mbox{ for any } \psi \in V, \mbox{ and a.e. }t \in (0,\sigma T).
      \end{gathered}
  \end{gathered} \right. \label{ST_01}
\end{equation}
and
\begin{equation}
  \left\{ \begin{gathered}
    \begin{gathered}
      \bigl((\sigma \partial_t p_{\varepsilon,\mu,\nu}^\circ + g'(\eta_{\varepsilon,\mu,\nu}^\circ)p_{\varepsilon,\mu,\nu}^\circ + \alpha''(\eta_{\varepsilon,\mu,\nu}^\circ) \gamma_\varepsilon(\nabla \theta_{\varepsilon,\mu,\nu}^\circ))(t), \varphi\bigr)_H 
      \\
      + ((\alpha'(\eta_{\varepsilon,\mu,\nu}^\circ) \nabla \gamma_\varepsilon(\nabla \theta_{\varepsilon,\mu,\nu}^\circ) \cdot \nabla z_{\varepsilon,\mu,\nu}^\circ)(t), \varphi)_H + (\nabla (p_{\varepsilon,\mu,\nu}^\circ + \sigma \mu^2 \partial_t p_{\varepsilon,\mu,\nu}^\circ)(t), \nabla \varphi)_H 
      \\
      = (\partial_t u^\circ(t), \varphi)_H, \ \mbox{ for any } \varphi \in V, \ \mbox{ and a.e. } t \in (0,\sigma T),
    \end{gathered}
    \\[2ex]
    \begin{gathered}
      \bigl( \sigma (\alpha_0(\eta_{\varepsilon,\mu,\nu}^\circ) \partial_t z_{\varepsilon,\mu,\nu}^\circ + \alpha_0'(\eta_{\varepsilon,\mu,\nu}^\circ) p_{\varepsilon,\mu,\nu}^\circ z_{\varepsilon,\mu,\nu}^\circ)(t), \psi \bigr)_H 
      \\
      + \bigl( (\alpha(\eta_{\varepsilon,\mu,\nu}^\circ) \nabla^2 \gamma_\varepsilon(\nabla \theta_{\varepsilon,\mu,\nu}^\circ) \nabla z_{\varepsilon,\mu,\nu}^\circ + \kappa \nabla z_{\varepsilon,\mu,\nu}^\circ + \sigma \nu^2 \nabla \partial_t z_{\varepsilon,\mu,\nu}^\circ)(t), \nabla \psi \bigr)_{[H]^N} 
      \\
      + \bigl( (\alpha'(\eta_{\varepsilon,\mu,\nu}^\circ ) p_{\varepsilon,\mu,\nu}^\circ \nabla \gamma_\varepsilon(\nabla \theta_{\varepsilon,\mu,\nu}^\circ))(t), \nabla \psi \bigr)_{[H]^N} = (\partial_t v^\circ(t), \psi)_H,
      \\
      \mbox{ for any } \psi \in V, \ \mbox{ and a.e. } t \in (0,\sigma T),
    \end{gathered}
  \end{gathered} \right. \label{ST_02}
\end{equation}
  with initial conditions:
  \begin{gather}
    [\eta_{\varepsilon,\mu,\nu}^\circ(0),\theta_{\varepsilon,\mu,\nu}^\circ(0)] := [\eta_0, \theta_{0,\varepsilon}] \mbox{ in } [W_0]^2, \mbox{ and} \label{ST_03}
    \\
    \left\{ \begin{aligned}
      p_{\varepsilon,\mu,\nu}^\circ(0) &= \sigma^{-1}(I - \mu^2 \Lap_N)^{-1}\bigl( \Lap_N \eta_0 - g(\eta_0) - \alpha'(\eta_0) \gamma_\varepsilon(\nabla \theta_{0,\varepsilon}) + u(0) \bigr) \in W_0,
      \\
      z_{\varepsilon,\mu,\nu}^\circ(0) &= \sigma^{-1}(\alpha_0(\eta_0) I - \nu^2 \Lap_N)^{-1} \bigl( -w_\varepsilon^* + v(0) \bigr)
      \\
      &= \sigma^{-1}(\alpha_0(\eta_0) I - \nu^2 \Lap_N)^{-1} \bigl( \diver \bigl( \alpha(\eta_0) \nabla \gamma_\varepsilon(\nabla \theta_{0,\varepsilon}) + \kappa \nabla \theta_{0,\varepsilon} \bigr) + v(0) \bigr) \in V.
    \end{aligned} \right.
  \end{gather}
  
  Now, we are ready to find the estimates associated with pseudo-parabolic type regularities to $[\eta_{\varepsilon,\mu,\nu},\theta_{\varepsilon,\mu,\nu}]$.
\begin{lemma}[Key estimates]\label{key_est}
  We set a positive constant $\sigma_* > 1$ as:
  \begin{equation}
    \sigma_* := \frac{8}{\kappa}|\alpha'|_{L^\infty}^2, \mbox{ so that } \ \frac{\sigma}{2} - \frac{2}{\kappa}|\alpha'|_{L^\infty}^2 \geq \frac{\sigma}{4} \mbox{ for } \sigma \geq \sigma_*.
  \end{equation}
  Then, there exists a positive constant $C_5 > 0$, independent of $\varepsilon, \mu, \nu$, such that the following estimate holds:
  \begin{align}
    &\quad J_{\varepsilon,\mu,\nu}(t) +\int_0^t \Bigl( |\nabla p_{\varepsilon,\mu,\nu}^\circ(t)|_{[H]^N}^2 + \kappa |\nabla z_{\varepsilon,\mu,\nu}^\circ(t)|_{[H]^N}^2 + \frac{\sigma}{2}|\partial_t p_{\varepsilon,\mu,\nu}^\circ(t)|_H^2
    \\
    &\quad\qquad+ \sigma \mu^2 |\nabla \partial_t p_{\varepsilon,\mu,\nu}^\circ(t)|_{[H]^N}^2 \Bigr) \,dt \label{ST_04}
    \\
    &\leq \frac{12}{\delta_\alpha \wedge 1} \exp\bigl( C_5 \bigl( |\theta_{\varepsilon,\mu,\nu}^\circ|_{L^2(0,\sigma T;H^2(\Omega))}^2 + |\partial_t \theta_{\varepsilon,\mu,\nu}^\circ|_{\sH_{\sigma T}}^2 + (\varepsilon^2 + 1) \sigma T \bigr)\bigr) \cdot 
    \\
    &\quad\qquad \cdot \bigl( J_{\varepsilon,\mu,\nu}(0) + |\partial_t u^\circ|_{\sH_{\sigma T}}^2 + |\partial_t v^\circ|_{\sH_{\sigma T}}^2 \bigr), \mbox{ for every } t \in [0,\sigma T], 
  \end{align}
  where
  \begin{gather}
    J_{\varepsilon,\mu,\nu}(t) := \sigma |p_{\varepsilon,\mu,\nu}^\circ(t)|_H^2 + (1 + \sigma \mu^2)|\nabla p_{\varepsilon,\mu,\nu}^\circ(t)|_{[H]^N}^2 + \sigma \bigl|\sqrt{\alpha_0(\eta_{\varepsilon,\mu,\nu}^\circ)(t)} z_{\varepsilon,\mu,\nu}^\circ(t) \bigr|_H^2
    \\
    + \sigma \nu^2 |\nabla z_{\varepsilon,\mu,\nu}^\circ(t)|_{[H]^N}^2, \mbox{ for } t \in [0,\sigma T], \label{ST_05}
  \end{gather}
\end{lemma}

\begin{proof}
  We let $\varphi = p_{\varepsilon,\mu,\nu}^\circ(t)$ in \eqref{ST_02}. By H\"{o}lder's and Young's inequalities, one can compute that
  \begin{gather}
    \frac{\sigma}{2} \frac{d}{dt}\bigl( |p_{\varepsilon,\mu,\nu}^\circ(t)|_H^2 + \mu^2 |\nabla p_{\varepsilon,\mu,\nu}^\circ(t)|_{[H]^N}^2 \bigr) + |\nabla p_{\varepsilon,\mu,\nu}^\circ(t)|_{[H]^N}^2 - |g'|_{L^\infty} |p_{\varepsilon,\mu,\nu}^\circ(t)|_H^2
    \\
    - \frac{\kappa}{8}|\nabla z_{\varepsilon,\mu,\nu}^\circ(t)|_{[H]^N}^2 - \frac{2}{\kappa}|\alpha'|_{L^\infty}^2 |p_{\varepsilon,\mu,\nu}^\circ(t)|_H^2 \leq \frac{1}{\sigma}|\partial_t u^\circ(t)|_H^2 + \frac{\sigma}{4}|p_{\varepsilon,\mu,\nu}^\circ(t)|_H^2, \label{ST_07}
    \\
    \mbox{ for a.e. } t \in (0,\sigma T). 
  \end{gather}

  Next, we consider letting $\psi = z_{\varepsilon,\mu,\nu}^\circ(t)$. Then, invoking the positivity of $\nabla^2 \gamma_\varepsilon$, we can estimate that
  \begin{align}
    &\sigma ((\alpha_0(\eta_{\varepsilon,\mu,\nu}^\circ) \partial_t z_{\varepsilon,\mu,\nu}^\circ + \alpha_0'(\eta_{\varepsilon,\mu,\nu}^\circ) p_{\varepsilon,\mu,\nu}^\circ z_{\varepsilon,\mu,\nu}^\circ)(t), z_{\varepsilon,\mu,\nu}^\circ(t))_H + \kappa |\nabla z_{\varepsilon,\mu,\nu}^\circ(t)|_{[H]^N}^2 
    \\
    &\quad + \frac{\sigma \nu^2}{2}\frac{d}{dt} |\nabla z_{\varepsilon,\mu,\nu}^\circ(t)|_{[H]^N}^2 - \frac{\kappa}{8}|\nabla z_{\varepsilon,\mu,\nu}^\circ(t)|_{[H]^N}^2 - \frac{2}{\kappa}|\alpha'|_{L^\infty}^2|p_{\varepsilon,\mu,\nu}^\circ(t)|_H^2 \label{ST_08}
    \\
    &\quad \leq \frac{1}{\sigma\delta_\alpha} |\partial_t u^\circ(t)|_H^2 + \frac{\sigma}{4}|\sqrt{\alpha_0(\eta_{\varepsilon,\mu,\nu}^\circ(t))} z_{\varepsilon,\mu,\nu}^\circ(t)|_H^2, \ \mbox{ for a.e. } t \in (0,\sigma T).
  \end{align}
  The first term of \eqref{ST_08} can be computed by:
  \begin{align}
    &\quad \sigma ((\alpha_0(\eta_{\varepsilon,\mu,\nu}^\circ) \partial_t z_{\varepsilon,\mu,\nu}^\circ + \alpha_0'(\eta_{\varepsilon,\mu,\nu}^\circ) p_{\varepsilon,\mu,\nu}^\circ z_{\varepsilon,\mu,\nu}^\circ)(t), z_{\varepsilon,\mu,\nu}^\circ(t))_H
    \\
    &= \frac{\sigma}{2}\frac{d}{dt} \bigl|\sqrt{\alpha_0(\eta_{\varepsilon,\mu,\nu}^\circ(t))} z_{\varepsilon,\mu,\nu}^\circ(t) \bigr|_H^2 - \frac{\sigma}{2}((\alpha_0'(\eta_{\varepsilon,\mu,\nu}^\circ) p_{\varepsilon,\mu,\nu}^\circ z_{\varepsilon,\mu,\nu}^\circ)(t), z_{\varepsilon,\mu,\nu}^\circ(t))_H
    \\
    &\quad\quad + \sigma ((\alpha_0'(\eta_{\varepsilon,\mu,\nu}^\circ) p_{\varepsilon,\mu,\nu}^\circ z_{\varepsilon,\mu,\nu}^\circ)(t), z_{\varepsilon,\mu,\nu}^\circ(t))_H \label{ST_09}
    \\
    &\geq \frac{\sigma}{2}\frac{d}{dt} \bigl|\sqrt{\alpha_0(\eta_{\varepsilon,\mu,\nu}^\circ(t))} z_{\varepsilon,\mu,\nu}^\circ(t) \bigr|_H^2 - \frac{\kappa}{8}|z_{\varepsilon,\mu,\nu}^\circ(t)|_V^2 - \frac{(C_V^{L^4})^4 \sigma^2}{2 \kappa}|\alpha_0'|_{L^\infty}^2 |z_{\varepsilon,\mu,\nu}^\circ(t)|_H^2 |p_{\varepsilon,\mu,\nu}^\circ(t)|_V^2.
  \end{align}
  On account of \eqref{ST_09}, \eqref{ST_08} can be reduced as follows:
  \begin{align}
    &\frac{\sigma}{2}\frac{d}{dt} \bigl( \bigl|\sqrt{\alpha_0(\eta_{\varepsilon,\mu,\nu}^\circ(t))}z_{\varepsilon,\mu,\nu}^\circ(t)\bigr|_H^2 + \nu^2 |\nabla z_{\varepsilon,\mu,\nu}^\circ(t)|_{[H]^N}^2\bigr) + \frac{3\kappa}{4}|\nabla z_{\varepsilon,\mu,\nu}^\circ(t)|_{[H]^N}^2
    \\
    &\quad \leq \frac{\kappa}{8}|z_{\varepsilon,\mu,\nu}^\circ(t)|_H^2 + \frac{2}{\kappa} |\alpha_0'|_{L^\infty}^2 + \frac{(C_V^{L^4})^4 \sigma^2}{2 \kappa} |z_{\varepsilon,\mu,\nu}^\circ(t)|_H^2 |p_{\varepsilon,\mu,\nu}^\circ(t)|_V^2 \label{ST_10}
    \\
    &\quad\quad + \frac{1}{\sigma \delta_\alpha} |\partial_t u^\circ(t)|_H^2 + \frac{\sigma}{4}|\sqrt{\alpha_0(\eta_{\varepsilon,\mu,\nu}^\circ(t))} z_{\varepsilon,\mu,\nu}^\circ(t)|_H^2, \ \mbox{ for a.e. } t \in (0,\sigma T).
  \end{align} 

  Finally, let us set $\varphi = \partial_t p_{\varepsilon,\mu,\nu}^\circ(t)$ in \eqref{ST_02}. Using H\"{o}lder's and Young's inequalities, we find the following inequality:
  \begin{align}
    &\sigma |\partial_t p_{\varepsilon,\mu,\nu}^\circ(t)|_H^2 + \sigma \mu^2 |\nabla \partial_t p_{\varepsilon,\mu,\nu}^\circ(t)|_{[H]^N}^2 + \frac{1}{2}\frac{d}{dt}|\nabla p_{\varepsilon,\mu,\nu}^\circ(t)|_{[H]^N}^2
    \\
    &\quad -\frac{2}{\sigma} |g'|_{L^\infty}^2 |p_{\varepsilon,\mu,\nu}^\circ(t)|_H^2 - \frac{2 (C_V^{L^4})^4}{\sigma} |\alpha''|_{L^\infty}^2 |p_{\varepsilon,\mu,\nu}^\circ|_V^2 (\varepsilon^2 \L^N(\Omega) + |\nabla \theta_{\varepsilon,\mu,\nu}^\circ(t)|_{[V]^N}^2) 
    \\
    &\quad - \frac{\sigma}{4}|\partial_t p_{\varepsilon,\mu,\nu}^\circ(t)|_H^2 - \frac{\kappa}{8}|\nabla z_{\varepsilon,\mu,\nu}^\circ(t)|_{[H]^N}^2 - \frac{2}{\kappa} |\alpha'|_{L^\infty}^2 |\partial_t p_{\varepsilon,\mu,\nu}^\circ(t)|_H^2 \label{ST_11}
    \\
    &\quad \leq \frac{\sigma}{8}|\partial_t p_{\varepsilon,\mu,\nu}^\circ(t)|_H^2 + \frac{2}{\sigma} |\partial_t u^\circ(t)|_H^2, \ \mbox{ for a.e. } t \in (0,\sigma T).
  \end{align}

  Now, we set the constant $C_5$ as:
  \begin{equation}
    C_5 := \frac{6\sigma^2((C_V^{L^4})^4 + 1)(\L^N(\Omega) + 1)}{\kappa \wedge 1}( |g'|_{L^\infty}^2 + |\alpha'|_{L^\infty} + |\alpha''|_{L^\infty}^2 + 1 ), \label{ST_14}
  \end{equation}
  and define a positive-valued function $R_{\theta,\varepsilon}^* \in L^1(0,\sigma T)$ as:
  \begin{gather}
    R_{\theta,\varepsilon}(t) := |\theta_{\varepsilon,\mu,\nu}^\circ(t)|_{H^2(\Omega)}^2 + |\partial_t \theta_{\varepsilon,\mu,\nu}^\circ(t)|_H^2 + \varepsilon^2 + 1, \mbox{ for a.e. } t \in (0,\sigma T). \label{ST_13}
  \end{gather}
  Combining \eqref{ST_07}, \eqref{ST_10} and \eqref{ST_11}, we arrive at the following Gronwall type inequality:
  \begin{gather}
    \frac{d}{dt} J_{\varepsilon,\mu,\nu}(t) + |\nabla p_{\varepsilon,\mu,\nu}^\circ(t)|_{[H]^N}^2 + \kappa |\nabla z_{\varepsilon,\mu,\nu}^\circ(t)|_{[H]^N}^2 + \frac{\sigma}{2}|\partial_t p_{\varepsilon,\mu,\nu}^\circ(t)|_H^2 + \sigma \mu^2 |\nabla \partial_t p_{\varepsilon,\mu,\nu}^\circ(t)|_{[H]^N}^2
    \\
    \leq C_* R_{\theta,\varepsilon}(t) J_{\varepsilon,\mu,\nu}(t) + \frac{12}{\delta_\alpha \wedge 1} \bigl( |\partial_t u^\circ(t)|_H^2 + |\partial_t u^\circ(t)|_H^2 \bigr), \mbox{ for a.e. } t \in (0,\sigma T). \label{ST_12}
  \end{gather}
  Therefore, we can apply Gronwall's lemma to \eqref{ST_12} and obtain the desired result \eqref{ST_04}.
\end{proof}

\subsection{Limiting observations for a subsequence of $\{ [\eta_{\varepsilon,\mu,\nu}, \theta_{\varepsilon,\mu,\nu}] \}$}\label{sec:limit}

Hereafter, let $\sigma = \sigma_* =8|\alpha'|_{L^\infty}^2/\kappa$ as in Lemma \ref{key_est}. First, due to Lemma \ref{init}, it follows that:
  \begin{itemize}
    \item[$\sharp$ 0)] $\{ \theta_{0,\varepsilon}\,|\, \varepsilon \in (0,1) \}$ is bounded in $W_0$, $\{ \partial \Phi_\kappa^\varepsilon(\alpha(\eta_0);\theta_{0,\varepsilon}) \,|\, \varepsilon \in (0,1) \}$ is bounded in $H$, and $\{\F_\kappa^\varepsilon(\eta_0, \theta_{0,\varepsilon})\,|\,\varepsilon \in (0,1) \} \subset [0,\infty)$ is bounded.
  \end{itemize}
  Moreover, owing to \eqref{init_001} and Remark \ref{rem:resolvent}, we can estimate that:
  \begin{align}
    &\quad \sigma |p_{\varepsilon,\mu,\nu}^\circ(0)|_V^2 \leq |\Lap \eta_0 - g(\eta_0) - \alpha'(\eta_0) \gamma_\varepsilon(\nabla \theta_{0,\varepsilon}) + u(0)|_V^2
    \\
    &\leq 4 \bigl( |\Lap \eta_0|_V^2 + |g(\eta_0)|_V^2 + |\alpha'(\eta_0) \gamma_\varepsilon(\nabla \theta_{0,\varepsilon})|_V^2 + |u(0)|_V^2 \bigr)
    \\
    &\leq 16 \bigl( |\alpha'|_{L^\infty}^2 + |\alpha''|_{L^\infty}^2 + 1 \bigr)((C_V^{L^4})^4 + 1) \bigl(|\eta_0|_{H^3(\Omega)}^2 + |g(\eta_0)|_V^2 + |u(0)|_V^2 + 1 \bigr) \cdot
    \\
    &\qquad \cdot \bigl( |\gamma_\varepsilon(\nabla \theta_{0,\varepsilon})|_V^2 + 1 \bigr) \label{ST_15}
    \\
    &\leq 32C_\kappa^* \bigl( |\alpha'|_{L^\infty}^2 + |\alpha''|_{L^\infty}^2 + 1 \bigr)((C_V^{L^4})^4 + 1) \bigl(|\eta_0|_{H^3(\Omega)}^2 + |g(\eta_0)|_V^2 + |u(0)|_V^2 + 1 \bigr) \cdot
    \\
    &\qquad \cdot \left( |w^* + \theta_0|_H^2 + |\alpha(\eta_0)|_V^2 + \varepsilon^2 \L^N(\Omega)  + 1 \right), \ \mbox{ for any } \varepsilon \in (0,1],
  \end{align}
  and
  \begin{equation}
    \sigma \Bigl( \frac{\delta_\alpha}{2} |z_{\varepsilon,\mu,\nu}^\circ(0)|_H^2 + \nu^2 |\nabla z_{\varepsilon,\mu,\nu}^\circ(0)|_{[H]^N}^2 \Bigr) \leq \frac{1}{2\delta_\alpha} |-w_\varepsilon^* + v(0)|_H^2, \ \mbox{ for any } \varepsilon \in (0,1].
  \end{equation}
  Therefore, we have the following boundedness:
  \begin{itemize}
    \item [$\sharp$ 1)] $\{ p_{\varepsilon,\mu,\nu}^\circ(0) \,|\, \varepsilon \in (0,1),\, \mu \in (0,1),\, \nu \in (0,1) \}$ is bounded in $V$, $\{ z_{\varepsilon,\mu,\nu}^\circ(0) \,|\, \varepsilon \in (0,1),\, \mu \in (0,1),\, \nu \in (0,1) \}$ is bounded in $H$, and 
    \item [$\sharp$ 1')] $\{ \nu z_{\varepsilon,\mu,\nu}^\circ(0) \,|\, \varepsilon \in (0,1),\, \mu \in (0,1),\, \nu \in (0,1) \}$ is bounded in $V$.
  \end{itemize}
  Next, $\sharp$ 0), energy-inequality as in (S3)$_\varepsilon^{\mu,\nu}$ of Lemma \ref{pseudo-parabolic}, and the estimates given in Lemma \ref{lem:H2} lead to:
  \begin{itemize}
    \item[$\sharp$ 2)] $\{ \eta_{\varepsilon,\mu,\nu} \,|\, \varepsilon \in (0,21),\, \mu \in (0,1),\, \nu \in (0,1) \}$, $\{ \theta_{\varepsilon,\mu,\nu} \,|\, \varepsilon \in (0,1),\, \mu \in (0,1),\, \nu \in (0,1) \}$ are bounded in $W^{1,2}(0,T;H)$ and in $L^2(0,T;W_0)$, and 
    \item[$\sharp$ 2')] $\{ \mu \eta_{\varepsilon,\mu,\nu} \,|\, \varepsilon \in (0,1),\, \mu \in (0,1),\, \nu \in (0,1) \}$, $\{ \nu \theta_{\varepsilon,\mu,\nu} \,|\, \varepsilon \in (0,1),\, \mu \in (0,1),\, \nu \in (0,1) \}$ are bounded in $W^{1,2}(0,T;V)$ and in $L^\infty(0,T;W_0)$.
  \end{itemize}
  Finally, Lemma \ref{key_est} enables us to find that
  \begin{itemize}
    \item[$\sharp$ 3)] $\{ \partial_t \eta_{\varepsilon,\mu,\nu} \,|\, \varepsilon \in (0,1),\, \mu \in (0,1),\, \nu \in (0,1) \}$ is bounded in $W^{1,2}(0,T;H)$ and in $L^\infty(0,T;V)$, and $\{ \partial_t \theta_{\varepsilon,\mu,\nu} \,|\, \varepsilon \in (0,1),\, \mu \in (0,1),\, \nu \in (0,1) \}$ is bounded in $L^\infty(0,T;H)$ and in $\sV_T$.
  \end{itemize}
  Here, applying Aubin's compactness theory (cf. \cite[Corollary 4]{MR0916688}), we can obtain a subsequence of $\{ \varepsilon_n \}_{n=1}^\infty \subset (0,1)$; $\varepsilon_n \to \varepsilon_0$, $\{ \mu_n \}_{n=1}^\infty \subset (0,1)$; $\mu_n \downarrow 0$ and $\{ \nu_n \}_{n=1}^\infty \subset (0,1)$; $\nu_n \downarrow 0$ with a limit $[\eta, \theta] \in [\sH_T]^2$ such that the following convergences hold as $n \to \infty$:
  \begin{equation}
    \left\{ \begin{gathered}
    \eta_n := \eta_{\varepsilon_n,\mu_n,\nu_n}\to \eta  \mbox{ in } C([0,T];H), \, \sV_T, \, \mbox{ and weakly in } W^{2,2}(0,T;H), \, L^2(0,T;W_0),
    \\
    \partial_t \eta_n \to \partial_t \eta \mbox{ in } C([0,T];H), \mbox{ weakly in } W^{1,2}(0,T;H), \mbox{ and weakly-$*$ in } L^\infty(0,T;V),
  \end{gathered} \right. \label{eta_conv}
  \end{equation}
  \begin{equation}
    \left\{ \begin{gathered}
    \theta_n := \theta_{\varepsilon_n,\mu_n,\nu_n} \to \theta \mbox{ in } C([0,T];H), \, \sV_T, \mbox{ and weakly in } W^{1,2}(0,T;V), \, L^2(0,T;W_0),
    \\
    \partial_t \theta_n \to \partial_t \theta \mbox{ weakly in } \sV_T, \mbox{ and weakly-$*$ in } L^\infty(0,T;H).
  \end{gathered} \right. \label{theta_conv}
  \end{equation}
  and
  \begin{equation}
    \eta_n(t) \to \eta(t), \ \theta_n(t) \to \theta(t) \mbox{ in } V, \mbox{ for a.e. } t \in [0,T], \label{pointwise_V_conv}
  \end{equation}
  with
  \begin{equation}
    [\eta(0), \theta(0)] = \lim_{n \to \infty} [\eta_n(0), \theta_n(0)] = \lim_{n \to \infty} [\eta_0, \theta_{0,\varepsilon_n}] = [\eta_0, \theta_0] \mbox{ in } [H]^2.
  \end{equation}
  We note that the statement $\eta, \theta \in L^\infty(0,T;W_0)$ has not been confirmed.

\subsection{Verification of (S0)--(S4)} \label{sec:sol}
  
  Now, $[\eta_n, \theta_n]$ solves the following variational system for any open interval $I \subset (0,T)$:
  \begin{gather}
      \int_I \bigl( \partial_t \eta_n(t) + g(\eta_n(t)) + \alpha'(\eta_n(t))\gamma_{\varepsilon_n}(\nabla \theta_n(t)), \varphi \bigr)_H \,dt \label{eta_VI1}
      \\
      + \int_I (\nabla (\eta_n + \mu_n^2 \partial_t \eta_n)(t), \nabla \varphi )_{{[H]^N}} \,dt = \int_I (u(t), \varphi)_H \,dt, \mbox{ for any } \varphi \in V,
    \end{gather}
    and
    \begin{gather}
        \int_I \bigl(\alpha_0(\eta_n(t)) \partial_t \theta_n(t), \theta_n(t) - \psi\bigr)_H \,dt + \int_I \Phi_\kappa^{\varepsilon_n}(\alpha(\eta_n(t));\theta_n(t)) \,dt \label{theta_VI1}
        \\
        + \int_I \bigl( \nu_n^2 \nabla \partial_t \theta_n(t), \nabla (\theta_n(t) - \psi) \bigr)_{[H]^N} \,dt \leq \int_I \Phi_\kappa^{\varepsilon_n}(\alpha(\eta_n(t));\psi) + \int_I (v(t), \psi)_H \,dt,
        \\
        \mbox{ for any } \psi \in V.
    \end{gather}
    According to \eqref{eta_conv} and \eqref{theta_conv}, letting $n \to \infty$ in \eqref{eta_VI1} and \eqref{theta_VI1} yields that:
    \begin{gather}
      \int_I \bigl( \partial_t \eta(t) + g(\eta(t)) + \alpha'(\eta(t))\gamma_{\varepsilon_0}(\nabla \theta(t)), \varphi \bigr)_H \,dt + \int_I (\nabla \eta(t), \nabla \varphi )_{{[H]^N}} \,dt 
      \\
      = \int_I (u(t), \varphi)_H \,dt, \mbox{ for any } \varphi \in V, \label{eta_VI2}
    \end{gather}
    and
    \begin{gather}
        \int_I \bigl(\alpha_0(\eta(t)) \partial_t \theta(t), \theta(t) - \psi\bigr)_H\,dt + \int_I \Phi_\kappa^{\varepsilon_0}(\alpha(\eta(t));\theta(t)) \,dt \label{theta_VI2}
        \\
        \leq \int_I \Phi_\kappa^{\varepsilon_0}(\alpha(\eta(t));\psi) + \int_I (v(t), \theta(t) - \psi)_H \,dt,
        \\
        \mbox{ for any } \psi \in V.
    \end{gather}
    Since $I$ is arbitrary, $[\eta,\theta]$ solves the variational system for a.e. $t \in (0,T)$. Also, from the convergences \eqref{eta_conv}--\eqref{pointwise_V_conv}, letting $n \to \infty$ in the energy inequality (S4)$_\varepsilon^{\mu,\nu}$ yields that the energy inequality (S4) holds for a.e. $t \in (0,T)$. Furthermore, the mappings $t \in [0,T] \mapsto \eta(t) \in V$ and $t \in [0,T] \mapsto \theta(t) \in V$ are continuous, the energy inequality (S4) holds for every $t \in [0,T]$.

    Finally, we confirm the strong regularity of the limit $[\eta, \theta]$. We set:
    \begin{equation}
      \tilde u := -\partial_t \eta - g(\eta) - \alpha'(\eta) \gamma_{\varepsilon_0}(\nabla \theta) + \eta + u \in L^\infty(0,T;H),
    \end{equation}
    and
    \begin{equation}
      \tilde v := -\alpha_0(\eta) \partial_t \theta + \theta + v \in L^\infty(0,T;H).
    \end{equation}
    Noting \eqref{ev_eta}, \eqref{ev_theta} and the regularity theory of elliptic problem as in \eqref{SD_est}, we can confirm the assertion $\eta, \theta \in L^\infty(0,T;W_0)$ owing to the following identity:
    \begin{equation}
      - \Lap_N \eta(t) + \eta(t) = \tilde u(t) \mbox{ in } H, \ \mbox{ for a.e. } t \in (0,T), \label{ev_eta2}
    \end{equation}
    and
    \begin{equation}
      \partial \Phi_\kappa^{\varepsilon_0}(\alpha(\eta(t));\theta(t)) + \theta(t) = \tilde v(t) \mbox{ in } H, \ \mbox{ for a.e. } t \in (0,T). \label{ev_theta2}
    \end{equation}
    with the following estimates:
    \begin{align}
      |\eta(t)|_{H^2(\Omega)}^2 &\leq C_0 |-\Lap_N \eta(t) + \eta(t)|_H^2
      \\
      &= C_0|-\partial_t \eta - g(\eta) - \alpha'(\eta) \gamma_{\varepsilon_0}(\nabla \theta) + \eta + u|_{L^\infty(0,T;H)}^2,
      \\
      &\qquad\qquad\qquad \mbox{ for a.e. } t \in (0,T),
    \end{align}
    and
    \begin{align}
      |\theta(t)|_{H^2(\Omega)}^2 &\leq | (\partial \Phi_\kappa^{\varepsilon_0}(\alpha(\eta(t)); \cdot) + I)^{-1} \tilde v|_H^2
      \\
      &\leq C_\kappa^*(|-\alpha_0(\eta) \partial_t \theta + \theta + v|_{L^\infty(0,T;H)}^2 + |\alpha(\eta)|_{L^\infty(0,T;V)}^2)  \label{theta_H2},
      \\
      &\qquad\qquad\qquad \mbox{ for a.e. } t \in (0,T).
    \end{align}
    
    Thus, we finish the verification of (S0)--(S4), and conclude that $[\eta, \theta]$ is the solution to (S)$_{\varepsilon_0}$.

\subsection{Verification of the uniqueness of solution.} \label{sec:uni}
    Let $[\eta_1, \theta_1]$ and $[\eta_2, \theta_2]$ be two solutions of (S)$_{\varepsilon_0}$ corresponding to the same initial data $[\eta_0,\theta_0] \in W_*^{\varepsilon_0}$ and the same forcing pair $[u,v]$. By letting $\varphi = (\eta_1 - \eta_2)(t) \in V$ in the variational identity (S1), and subtracting the both sides of variational identity (S1), we can compute that:
    \begin{gather}
      \frac{1}{2}\frac{d}{dt}(|(\eta_1 - \eta_2)(t)|_H^2) + |\nabla (\eta_1 - \eta_2)(t)|_{[H]^N}^2 - |g'|_{L^\infty} |(\eta_1 - \eta_2)(t)|_H^2 \label{eta_uni1}
      \\
      + (\alpha'(\eta_1(t)) \gamma_{\varepsilon_0}(\nabla \theta_1(t)) - \alpha'(\eta_2(t)) \gamma_{\varepsilon_0}(\nabla \theta_2(t)), (\eta_1 - \eta_2)(t))_H = 0, \mbox{ for a.e. } t \in (0,T).
    \end{gather}
    Noting the convexity of $\alpha$, we see that:
    \begin{align}
      &(\alpha'(\eta_1(t)) \gamma_{\varepsilon_0}(\nabla \theta_1(t)) - \alpha'(\eta_2(t)) \gamma_{\varepsilon_0}(\nabla \theta_2(t)), (\eta_1 - \eta_2)(t))_H
      \\
      &\quad \geq \int_\Omega (\alpha(\eta_1(t)) - \alpha(\eta_2(t))) \gamma_{\varepsilon_0}(\nabla \theta_1(t)) \,dx + \int_\Omega (\alpha(\eta_2(t)) - \alpha(\eta_1(t))) \gamma_{\varepsilon_0}(\nabla \theta_2(t))\,dx
      \\
      &\quad = \int_\Omega (\alpha(\eta_1(t)) - \alpha(\eta_2(t))) (\gamma_{\varepsilon_0}(\nabla \theta_1(t)) - \gamma_{\varepsilon_0}(\nabla \theta_2(t))) \,dx
      \\
      &\quad \geq - \frac{2}{\kappa} |\alpha'|_{L^\infty}^2 |(\eta_1 - \eta_2)(t)|_H^2 - \frac{\kappa}{8} |\nabla (\theta_1 - \theta_2)(t)|_{[H]^N}^2, \mbox{ for a.e. } t \in (0,T). \label{eta_uni2}
    \end{align}
    By using \eqref{eta_uni2}, \eqref{eta_uni1} can be reduced to:
    \begin{align}
      &\frac{1}{2}\frac{d}{dt}(|(\eta_1 - \eta_2)(t)|_H^2) + |\nabla (\eta_1 - \eta_2)(t)|_{[H]^N}^2 \label{eta_uni3}
      \\
      &\quad \leq \Bigl(|g'|_{L^\infty} + \frac{2}{\kappa}|\alpha'|_{L^\infty}^2 \Bigr) |(\eta_1 - \eta_2)(t)|_H^2 + \frac{\kappa}{8}|\nabla(\theta_1 - \theta_2)(t)|_{[H]^N}^2, \mbox{ for a.e. } t \in (0,T).
    \end{align}
    On the other hand, by letting $\psi = \theta_2(t)$ (resp. $\psi = \theta_1(t)$) in the variational inequality for $\theta_1$ (resp. $\theta_2$), and subtracting the both sides of variational inequality (S2), we have:
    \begin{align}
      &\quad (\alpha_0(\eta_1(t)) \partial_t \theta_1(t) - \alpha_0(\eta_2(t)) \partial_t \theta_2(t), \theta_1(t) - \theta_2(t))_H + \kappa|\nabla (\theta_1 - \theta_2)(t)|_{[H]^N}^2
      \\
      &\leq \int_\Omega \Bigl( \alpha(\eta_1(t))(\gamma_{\varepsilon_0}(\nabla\theta_1(t)) - \gamma_{\varepsilon_0}(\nabla \theta_2(t))) -\alpha(\eta_2(t)) (\gamma_{\varepsilon_0}(\nabla \theta_1(t)) - \gamma_{\varepsilon_0}(\nabla \theta_2(t))) \Bigr) \,dx
      \\
      &\leq \int_\Omega ( \alpha(\eta_1(t)) - \alpha(\eta_2(t))) (\gamma_{\varepsilon_0}(\nabla \theta_1(t)) - \gamma_{\varepsilon_0}(\nabla \theta_2(t))) \,dx
      \\
      &\leq \frac{2}{\kappa} |\alpha'|_{L^\infty}^2 |(\eta_1 - \eta_2)(t)|_H^2 + \frac{\kappa}{8} |\nabla (\theta_1 - \theta_2)(t)|_{[H]^N}^2, \mbox{ for a.e. } t \in (0,T). \label{theta_uni1}
    \end{align}
   Taking account of the continuous embedding from $V$ to $L^4(\Omega)$, which is valid when $N \leq 4$, we can compute the first term of \eqref{theta_uni1} as follows:
    \begin{align}
      &\quad (\alpha_0(\eta_1(t)) \partial_t \theta_1(t) - \alpha_0(\eta_2(t)) \partial_t \theta_2(t), \theta_1(t) - \theta_2(t))_H
      \\
      &= \frac{1}{2}\frac{d}{dt} \bigl( |\sqrt{\alpha_0(\eta_1(t))} (\theta_1 - \theta_2)(t)|_H^2\bigr) - \frac{1}{2} \int_\Omega \alpha_0'(\eta_1(t)) \partial_t \eta_1(t) |(\theta_1 - \theta_2)(t)|^2 \,dx
      \\
      &\qquad + \int_\Omega (\alpha_0(\eta_1(t)) - \alpha_0(\eta_2(t))) \partial_t \theta_2(t) (\theta_1 - \theta_2)(t) \,dx \label{theta_uni2}
      \\
      &= \frac{1}{2}\frac{d}{dt} \bigl( |\sqrt{\alpha_0(\eta_1(t))} (\theta_1 - \theta_2)(t)|_H^2\bigr) 
      \\
      &\qquad - \frac{1}{2} |\alpha_0'|_{L^\infty} |\partial_t \eta_1(t)|_{L^4(\Omega)} |(\theta_1 - \theta_2)(t)|_H |(\theta_1 - \theta_2)(t)|_{L^4(\Omega)}
      \\
      &\qquad - |\alpha_0'|_{L^\infty}|(\eta_1 - \eta_2)(t)|_H |\partial_t \theta_2(t)|_{L^4(\Omega)}|(\theta_1 - \theta_2)(t)|_{L^4(\Omega)}
      \\
      & \leq \frac{1}{2}\frac{d}{dt} \bigl( |\sqrt{\alpha_0(\eta_1(t))} (\theta_1 - \theta_2)(t)|_H^2\bigr) - \frac{\kappa}{4}|(\theta_1 - \theta_2)(t)|_V^2
      \\
      &\qquad - \frac{2 (C_V^{L^4})^4 |\alpha_0'|_{L^\infty}^2}{\kappa}(|\partial_t \eta_1(t)|_V^2 + |\partial_t \theta_2(t)|_V^2)(|(\eta_1 - \eta_2)(t)|_H^2 + |(\theta_1 - \theta_2)(t)|_H^2).
    \end{align}
    Hence, combining \eqref{theta_uni1} and \eqref{theta_uni2}, we have
    \begin{align}
      &\frac{1}{2}\frac{d}{dt} \bigl( |\sqrt{\alpha_0(\eta_1(t))} (\theta_1 - \theta_2)(t)|_H^2\bigr) + \frac{5\kappa}{8}|\nabla (\theta_1 - \theta_2)(t)|_{[H]^N}^2
      \\
      &\quad \leq \biggl( \frac{2}{\kappa}((C_V^{L^4})^4 |\alpha_0'|_{L^\infty}^2 + |\alpha'|_{L^\infty}) + \frac{\kappa}{4} \biggr)(|\partial_t \eta_1(t)|_V^2 + |\partial_t \theta_2(t)|_V^2 + 1)\cdot \label{theta_uni3}
      \\
      &\qquad\quad \cdot (|(\eta_1 - \eta_2)(t)|_H^2 + |(\theta_1 - \theta_2)(t)|_H^2), \mbox{ for a.e. } t \in (0,T).
    \end{align}

    Now, we set:
    \begin{gather}
      J^*(t) := |(\eta_1 - \eta_2)(t)|_H^2 + | \sqrt{\alpha_0(\eta_1(t))}(\theta_1 - \theta_2)(t)|_H^2, \ \mbox{ for } t \in [0,T],
      \\
      R^*(t) := |\partial_t \eta_1(t)|_V^2 + |\partial_t \theta_2(t)|_V^2 + 1, \ \mbox{ for a.e. } t \in (0,T),
    \end{gather}
    and
    \begin{equation}
      \widetilde{C}_6 := \frac{4}{(\kappa \wedge 1)(\delta_\alpha \wedge 1)} (|g'|_{L^\infty} + |\alpha'|_{L^\infty}^2 + (C_V^{L^4})^4|\alpha_0'|_{L^\infty}^2 + \kappa).
    \end{equation}
    Then, by virtue of \eqref{eta_uni3} and \eqref{theta_uni3}, we arrive at the following Gronwall-type inequality:
    \begin{gather}
      \frac{1}{2} \frac{d}{dt} J^*(t) + |\nabla(\eta_1 - \eta_2)(t)|_{[H]^N}^2 + \frac{\kappa}{2}|\nabla (\theta_1 - \theta_2)(t)|_{[H]^N}^2 \leq \widetilde{C}_6 R_*(t) J^*(t),
      \\
      \mbox{ for a.e. } t \in (0,T),
    \end{gather}
    which implies that the solution is unique, and thus, we finish the proof of Main Theorem \ref{mainThm1}.

\subsection{The proof of Main Theorem \ref{mainThm2}} \label{sec:mainThm2}

Thanks to limiting observations in the estimates as in Lemma \ref{lem:H2}, \ref{init} and \ref{key_est}, we have the following corollary, which supports the proof of Main Theorem \ref{mainThm2}.

\begin{cor}\label{cor:est}
  The unique solution $[\eta, \theta]$ to (S)$_{\varepsilon_0}$ fulfills the following estimates:
  \begin{gather}
      \left\{ \begin{aligned}
    &\frac{1}{2} |\Lap \eta|_{\sH_T}^2 \leq |\nabla \eta_0|_{[H]^N}^2 + C_2(\varepsilon_0^2 + 1)\bigl( |\eta|_{\sH_T}^2 + |\nabla \theta|_{[\sH_T]^N}^2 + |u|_{\sH_T}^2 + T \bigr),
    \\
    &\frac{\kappa}{2} |\Lap \theta|_{\sH_T}^2 \leq \frac{4C_3}{\delta_\alpha^2 \wedge 1}(|\eta|_{\sV_T}^2 + |\theta|_{\sH_T}^2 + \sF_\kappa^{\varepsilon_0}(\eta_0, \theta_0) + |u|_{\sH_T}^2 + |v|_{\sH_T}^2 + T),
  \end{aligned} \right.
  \end{gather}
  and
  \begin{align}
    &\quad |\partial_t \eta|_{L^\infty(0,T;V)}^2 + |\partial_t \theta|_{L^\infty(0,T;H)}^2 + |\nabla \partial_t \eta|_{[\sH_T]^N}^2 + \kappa |\nabla \partial_t \theta|_{[\sH_T]^N}^2 + |\partial_t^2 \eta|_{\sH_T}^2
    \\
    &\leq \sigma C_7 \exp\biggl( \frac{24(C_0 + 1)C_5 C_3 \sigma}{(\kappa \wedge 1)(\delta_\alpha \wedge 1)} \cdot
    \\
    &\qquad\qquad\qquad \quad \cdot (|\eta|_{\sV_T}^2 + |\theta|_{\sH_T}^2 + \sF_\kappa^{\varepsilon_0}(\eta_0, \theta_0) + |u|_{\sH_T}^2 + |v|_{\sH_T}^2 + (\varepsilon_0^2 + 2)T) \biggr) \cdot
    \\
    &\qquad \cdot \bigl( |\eta_0|_{H^3(\Omega)}^2 + |g(\eta_0)|_V^2 + |u(0)|_V^2 + |\partial_t u|_{\sH_T}^2 + 1 \bigr) \cdot 
    \\
    &\qquad \cdot \bigl( |\theta_0^*|_H^2 + |\theta_0|_V^2 + |\alpha(\eta_0)|_V^2 + |v(0)|_H^2 + |\partial_t v|_{\sH_T}^2 + \varepsilon_0^2 \L^N(\Omega) + 1 \bigr),
  \end{align}
  where $\theta_0^* \in \partial \Phi_\kappa^{\varepsilon_0}(\alpha(\eta_0);\theta_0)$ is arbitrary, and $C_7$ is a constant, given as:
  \begin{equation}
    C_7 := \frac{1536 C_\kappa^*}{\delta_\alpha^2 \wedge 1}(|\alpha''|_{L^\infty}^2 + |\alpha'|_{L^\infty}^2 + 1)((C_V^{L^4})^4 + 1).
  \end{equation}  
\end{cor}

\begin{proof}[The proof of Main Theorem \ref{mainThm2}]
    The proof is based on the limiting observation as in subsection \ref{sec:limit}. According to the assumptions \eqref{mainThm2_01}, the energy-inequality, one can say that:
    \begin{itemize}
      \item[$\flat$ 0)] $\{ \F_\kappa^{\varepsilon_n}(\eta_{0,n}, \theta_{0,n}) \,|\, n = 1,2,3,\dots \} \subset [0,\infty)$ is bounded.
      \item[$\flat$ 1)] $\{ [\eta_n, \theta_n] \,|\, n = 1,2,3,\dots \}$ is bounded in $[W^{1,2}(0,T;H)]^2$ and in $[L^\infty(0,T;V)]^2$.
    \end{itemize}
    
    Next, according to Corollary \ref{cor:est}, the solutions $[\eta_n, \theta_n]$ fulfills the following estimates:
    \begin{gather}
      \left\{ \begin{aligned}
    &\frac{1}{2} |\Lap \eta_n|_{\sH_T}^2 \leq |\nabla \eta_{0,n}|_{[H]^N}^2 + C_2(\varepsilon_n^2 + 1)\bigl( |\eta_n|_{\sH_T}^2 + |\nabla \theta_n|_{[\sH_T]^N}^2 + |u_n|_{\sH_T}^2 + T \bigr),
    \\
    &\frac{\kappa}{2} |\Lap \theta_n|_{\sH_T}^2 \leq \frac{4C_3}{\delta_\alpha^2 \wedge 1}(|\eta_n|_{\sV_T}^2 + |\theta_n|_{\sH_T}^2 + |\partial_t \theta_n|_{\sH_T}^2 + |v_n|_{\sH_T}^2 + T),
  \end{aligned} \right.
  \end{gather}
  and
  \begin{align}
    &|\partial_t \eta_n|_{L^\infty(0,T;V)}^2 + |\partial_t \theta_n|_{L^\infty(0,T;H)}^2 + |\nabla \partial_t \eta_n|_{[\sH_T]^N}^2 + \kappa |\nabla \partial_t \theta_n|_{[\sH_T]^N}^2 + |\partial_t^2 \eta_n|_{\sH_T}^2
    \\
    &\leq \sigma C_7 \exp\biggl( \frac{24(C_0 + 1)C_5 C_3 \sigma}{(\kappa \wedge 1)(\delta_\alpha \wedge 1)} \cdot
    \\
    &\qquad\qquad\qquad \quad \cdot (|\eta_n|_{\sV_T}^2 + |\theta_n|_{\sH_T}^2 + \sF_\kappa^{\varepsilon_n}(\eta_{0_n}, \theta_{0_n}) + |u_n|_{\sH_T}^2 + |v_n|_{\sH_T}^2 + (\varepsilon_n^2 + 2)T) \biggr) \cdot
    \\
    &\qquad \cdot \bigl( |\eta_{0_n}|_{H^3(\Omega)}^2 + |g(\eta_{0,n})|_V^2 + |u_n(0)|_V^2 + |\partial_t u_n|_{\sH_T}^2 + 1 \bigr) \cdot 
    \\
    &\qquad \cdot \bigl( |\theta_{0,n}^*|_H^2 + |\theta_{0,n}|_V^2 + |\alpha(\eta_{0,n})|_V^2 + |v_n(0)|_H^2 + |\partial_t v_n|_{\sH_T}^2 + \varepsilon_n^2 \L^N(\Omega) + 1 \bigr).
  \end{align}  
  Therefore, by combining the energy-inequality, one can see that the sequence of solutions $\{ [\eta_n, \theta_n] \}_{n=1}^\infty$ has the following boundedness:
  \begin{itemize}
    \item[$\flat$ 2)] $\{ \eta_n \,|\, n = 1,2,3,\dots \}$ is bounded in $W^{2,2}(0,T;H)$, $W^{1,\infty}(0,T;V)$, and in $L^2(0,T;W_0)$, and $\{ \theta_n \,|\, n = 1,2,3,\dots \}$ is bounded in $W^{1,\infty}(0,T;H)$, $W^{1,2}(0,T;V)$, and \linebreak in $L^2(0,T;W_0)$.
  \end{itemize}
  Finally, by the same arguments as in \eqref{ev_eta2}--\eqref{ev_theta2}, we can arrive at the following estimates:
  \begin{itemize}
    \item[$\flat$ 3)] $\{ [\eta_n, \theta_n] \,|\, n = 1,2,3,\dots \}$ is bounded in $[L^\infty(0,T;W_0)]^2$.
  \end{itemize}

  Thus, applying the compactness theory of Aubin's type again, we can find a subsequence $\{ [\eta_{n_k}, \theta_{n_k}] \}_{k=1}^\infty$ with a limiting pair $[\tilde \eta, \tilde \theta]$ such that
  \begin{gather}
    \left\{ \begin{aligned}
      &\eta_{n_k} \to \tilde \eta \mbox{ in } C([0,T];V), \mbox{ weakly in } W^{2,2}(0,T;H), \, W^{1,q}(0,T;V) \mbox{ for any } q \in [1,\infty),
      \\
      &\partial_t \eta_{n_k} \to \partial_t \tilde \eta \mbox{ in } C([0,T];H), \mbox{ weakly in } W^{1,2}(0,T;H) \mbox{ and weakly-$*$ in } L^\infty(0,T;V),
      \\
      &\theta_{n_k} \to \tilde \theta \mbox{ in } C([0,T];V), \mbox{ weakly in } W^{1,2}(0,T;V), \, W^{1,q}(0,T;H), \mbox{ for any } q \in [1,\infty),
      \\
      &\partial_t \theta_{n_k} \to \partial_t \tilde \theta \mbox{ weakly-$*$ in } L^\infty(0,T;H),
    \end{aligned} \right.
    \\
    \mbox{ as } k \to \infty.
  \end{gather}
  Moreover, the limiting process as in subsection \ref{sec:sol} leads to that $[\tilde \eta, \tilde \theta]$ is the solution to the system (S)$_{\varepsilon_0}$ with the initial data $[\eta_0, \theta_0]$ and the forcing pair $[u,v]$. Since the solution is unique, we can conclude that $[\tilde \eta, \tilde \theta] = [\eta, \theta]$, and hence, the convergence \eqref{mainThm2_02} is verified for some subsequence of $\{ [\eta_n, \theta_n] \}_{n=1}^\infty$.

  Finally, we can say that the uniqueness of solution to (S)$_{\varepsilon_0}$ implies that the limit point of $\{ [\eta_n, \theta_n] \}_{n=1}^\infty$ is unique. Thus, we can conclude that the convergences \eqref{mainThm2_02} hold without taking a subsequence, and hence, the proof of Main Theorem \ref{mainThm2} is completed.

  \end{proof}

  \medskip
  \noindent
  \textbf{Acknowledgements}: The authors are grateful to Professor Mitsuharu Otani at Waseda University for variable comments on this paper. The first author of this work is supported by JST SPRING, Grant Number JPMJSP2109.


\end{document}